\documentclass[10pt]{article}

\usepackage{anyfontsize,fix-cm}
\usepackage[english]{babel}
\usepackage{amssymb, amsthm, enumerate, mathtools,xparse,leftindex}
\usepackage{mathabx}
\usepackage[margin = 1.2in]{geometry}

\usepackage[backref, pdfencoding=auto,  psdextra]{hyperref}
\usepackage[abbrev, nobysame]{amsrefs} 
\makeatletter
\def\bibtex@style{amsrn-mod}
\makeatother
\BibSpec{arxiv}{%
    +{}  {\PrintAuthors}                {author}
    +{,} { \textit}                     {title}
    +{:} { \textit}                     {subtitle}
    +{,} { \PrintContributions}         {contribution}
    +{.} { \PrintPartials}              {partial}
    +{,} { }                            {journal}
    +{}  { \PrintDatePV}                {date}
     +{,} { \eprint}        {url}
}

\renewcommand{\MR}[1]{} 
\renewcommand{\PrintDOI}[1]{}

\usepackage{tikz-cd}
\usetikzlibrary{babel, patterns, calc}
\NewDocumentCommand{\tens}{e{_^}}{%
  \mathbin{\mathop{\otimes}\displaylimits
    \IfValueT{#1}{_{#1}}
    \IfValueT{#2}{^{#2}}
  }%
}

\newtheorem{Theorem}{Theorem}[section]

\newtheorem{corollary}[Theorem]{Corollary} 
\newtheorem{lemma}[Theorem]{Lemma}

\newtheorem*{torsionconj}{Torsion concentration conjecture}
\newtheorem*{kunneth}{K\"unneth Formula}
\newtheorem*{singerconjecture}{Singer conjecture}

\theoremstyle{definition} 
\newtheorem{definition}[Theorem]{Definition} 
\newtheorem{remark}[Theorem]{Remark}

\newtheorem*{notation}{Notation}

\newcommand{\onto}{\twoheadrightarrow}
\newcommand{\Dg}{D^{g}} 
\newcommand{\phig}{\phi^{g}}

\newcommand{\nsgp}{\mathrel{\lhd}} 
\newcommand{\cD}{\mathcal {D}} 
\newcommand{\cu}{\mathcal {U}} 
\newcommand{\F}{{\mathbb F}} 
\newcommand{\zz}{{\mathbb Z}} 
\newcommand{\Q}{{\mathbb Q}} 
\newcommand{\Fp}{{\mathbb F_p}} 
\newcommand{\rr}{{\mathbb R}} 

\newcommand{\DFH}{\cD_{\F H}}
\newcommand{\DFG}{\cD_{\F G}}
\newcommand{\DQ}{\cD_{\F Q}}
\newcommand{\DQG}{\cD_{\Q G}}
\newcommand{\DCL}{\cD_{\F C_{L}}}

\DeclareMathOperator{\Tor}{Tor}

\newcommand{\ga}{\alpha} 
\newcommand{\gb}{\beta} 
\newcommand{\gd}{\delta}

\newcommand{\gr}{\rho} 
\newcommand{\gs}{\sigma} 
 
\newcommand{\gt}{\tau}

\newcommand{\gD}{\Delta}

\newcommand{\gS}{\Sigma} 

\DeclareMathOperator{\Lk}{Lk} 
\DeclareMathOperator{\St}{St}
\DeclareMathOperator{\oSt}{\mathring{St}} 
\DeclareMathOperator{\Sub}{Sub} 
 
\newcommand{\medcup}{\textstyle \bigcup\limits}
\newcommand{\medcap}{\textstyle \bigcap\limits}

\newenvironment{enumerate1}{ 
\begin{enumerate}[\upshape (1)]}	
	{ 
\end{enumerate}
} 

\newenvironment{enumeratei}{ 
\begin{enumerate}[\upshape (i)]}	
	{ 
\end{enumerate}
} 

\makeatletter
\def\Ddots{\mathinner{\mkern1mu\raise\p@
\vbox{\kern7\p@\hbox{.}}\mkern2mu
\raise4\p@\hbox{.}\mkern2mu\raise7\p@\hbox{.}\mkern1mu}}
\makeatother

\renewcommand{\d}{\partial} 
\DeclarePairedDelimiter\abs{\lvert}{\rvert}
\newcommand{\into}{\hookrightarrow}
\DeclareMathOperator{\logtor}{logtor}

\numberwithin{equation}{section} 
\title{Edge subdivisions and the $L^2$-homology of right-angled Coxeter groups}
\begin{document}\author{Grigori Avramidi\thanks{The first author would like to thank the Max Planck Institut f\"ur Mathematik for its hospitality and financial support.}
\and Boris Okun\thanks{This work was initiated while the second author was visiting the Max Planck Institut f\"ur Mathematik in Bonn.
The second author would like to thank the MPIM for the financial support.}
\and Kevin Schreve\thanks{The third author was partially supported by the NSF grant DMS-2203325.} } \maketitle
\begin{abstract}
    If $L$ is a flag triangulation of $S^{n-1}$, then the Davis complex $\Sigma_L$ for the associated right-angled Coxeter group $W_L$ is a contractible $n$-manifold.
    A special case of a conjecture of Singer predicts that the $L^2$-homology of such $\Sigma_L$ vanishes outside the middle dimension.
    We give conditions which guarantee this vanishing is preserved under edge subdivision of $L$.
    In particular, we verify Singer's conjecture when $L$ is the barycentric subdivision of the boundary of an $n$-simplex, and for general barycentric subdivisions of triangulations of $S^{2n-1}$.
    Using this, we construct explicit counterexamples to a torsion growth analogue of Singer's conjecture.
\end{abstract}

\section{Introduction}

Suppose $Y$ is a CW-complex and $G$ a group which acts properly, cellularly, and cocompactly on $Y$.
One can define the reduced $L^2$-homology groups $L^2H_*(Y)$, which via a combinatorial Hodge decomposition can be identified with the Hilbert spaces of square summable harmonic chains of $Y$.
Using the group action, one can define associated dimensions for these homology groups, the $L^2$-Betti numbers $b_i^{(2)}(Y; G)$.

This paper is motivated by the following conjecture of Singer:
\begin{singerconjecture}
    Let $M^n$ be a closed aspherical manifold and $\widetilde M^n$ its universal cover.
    Then
    \[
    b_{\neq n/2}^{(2)}(\widetilde M^n ; \pi_1(M^n)) = 0.
    \]
\end{singerconjecture}

In \cite{do01}, Davis and the second author developed a program for proving the Singer conjecture for certain manifolds constructed using right-angled Coxeter groups.
In particular, if $L$ is any flag triangulation of $S^{n-1}$, then a construction of Davis outputs a closed, locally CAT(0) $n$-manifold $P_L$, whose fundamental group is the commutator subgroup of the associated right-angled Coxeter group $W_L$.
The universal cover of $P_L$ is the Davis complex $\Sigma_L$.
For these manifolds $P_L$, they proved the Singer conjecture in dimension $4$.

For triangulations of $S^{\ge 4}$, very little was known, though Davis and the second author showed the Singer conjecture for barycentric subdivisions of triangulations of $S^5$ and $S^7$ \cite{do04}.
In this paper, we extend this result to all even dimensions.
This was claimed by the second author in \cite{o04}, but there is a gap in the argument.
The crucial computation is a resolution of the Singer conjecture for the manifold $P_{b\d \Delta^n}$ where $b\d \Delta^n$ is the barycentric subdivision of the boundary of an $n$-simplex.
\begin{Theorem}\label{t:main}
    The Singer conjecture holds for $P_L$ if
    \begin{enumerate1}
        \item $L = b\d \Delta^n$ for all $n$.\label{t:mainitemone}

        \item $L$ is the barycentric subdivision of a triangulation of $S^{2n-1}$.\label{t:mainitemtwo}
    \end{enumerate1}
\end{Theorem}

Though the first item of Theorem \ref{t:main} may seem a bit artificial, it is used to prove the second item.
In fact, the manifolds $P_{b\d \Delta^n}$ come up in various ways in geometric topology/geometric group theory.
For example, Tomei \cite{t84} showed that the space $T^{n}$ of the real, symmetric, tridiagonal $(n+1) \times (n+1)$-matrices with fixed simple spectrum is a closed, aspherical manifold, and Davis \cite{d87} showed that $T^{n}$ is commensurable with $P_{b\d \Delta^n}$.
These also come up in various hyperbolization procedures.
Gromov introduced the ``M\"{o}bius band procedure'' and the ``product with interval procedure'', which inputted a cell complex (cubical in the first case) and outputted a nonpositively curved cube complex.
In \cite{cd95c}, Charney and Davis showed that the M\"{o}bius band procedure applied to $\d[0,1]^n$ and the product with interval procedure applied to $\d \Delta^n$ are again commensurable to $P_{b\d \Delta^n}$.

We now explain why computing the $L^2$-homology of $\Sigma_{bL}$ for barycentric subdivisions is easier than for general flag triangulations.
The program developed in \cite{do01} relied on Mayer--Vietoris arguments, where a flag triangulation $L$ is cut into pieces along links of vertices: $L = L-v \cup_{\Lk v} \St v$.
This leads to a decomposition of $\Sigma_L$ into pieces whose components are copies of $\Sigma_{L-v}$ and $\Sigma_{\St v}$ intersecting along copies of $\Sigma_{\Lk v}$.
For barycentric subdivisions, the link structure of vertices is well known: they decompose as joins of barycentric subdivisions of lower dimensional complexes.
Therefore, via K\"{u}nneth formulas in $L^2$-homology they often can be shown to have trivial $L^2$-homology, hence by Mayer--Vietoris we can remove them from $L$ without changing $b_\ast^{(2)}(\Sigma_L)$.
For barycentric subdivisions of $S^{2n-1}$, for $n \le 4$, there are enough vertices with $\Sigma_{\Lk v}$ $L^2$-acyclic such that the resulting complex has dimension $n-1$.
Therefore, the Davis manifolds have the same $L^2$-Betti numbers as an $n$-dimensional subcomplex; so in particular the $L^2$-homology of $\Sigma_L$ vanishes above the middle dimension and Poincar\'{e} duality implies it vanishes below the middle dimension.

However, this strategy breaks in odd dimensions.
For example, if $L = b\d \Delta^3$, then we can remove from $L$ all barycenters of edges, as their links are $4$-cycles, and $\Sigma_{\Lk v} \cong \rr^2$ which is $L^2$-acyclic.
We are left with the $1$-skeleton of a $3$-cube, there are no more vertices with $L^2$-acyclic links, and it is not obvious that the $L^2$-homology of the corresponding Davis complex vanishes.
In this specific case, the group acts on $\mathbb{H}^{3}$ by reflections in the faces of an ideal right-angled octahedron, which implies $L^2$-acyclicity by analytic methods, see e.g.
\cite{ll95}.
Or, one can use L\"{u}ck's mapping torus theorem \cite{l94a} as the RACG virtually fibers via the ``combinatorial game'' of Jankiewicz--Norin--Wise \cite{jnw21}, but (as far as we know) these proofs cannot be adapted to higher dimensions.

To handle this, the new tool comes from the skew field approach to $L^2$-Betti numbers, first introduced by Linnell for $\Q$-coefficients \cite{l93} and explored by Jaikin-Zapirain, Henneke--Kielak, and others for general $\F$-coefficients \cite{j21}, \cite{hk21}.
For a large class of groups $\Q G$ embeds into a skew field $\cD_{\Q G}$, and if $Y$ is a $G$-complex, one can compute the usual $L^2$-Betti numbers of $Y$ by taking the associated Betti numbers $b^G_\ast(Y; \cD_{\Q G})$ of the equivariant cohomology of $Y$ with $\cD_{\Q G}$ coefficients.
Moreover, Jaikin-Zapirain \cite{j21} showed that if $G$ is residually (locally indicable amenable), this embedding is \emph{universal in Cohn's sense}.
By definition, this says that for any skew field $D$ and homomorphism $\phi: \Q G \to D$, we have the inequality $b^G_*(Y;\cD_{\Q G}) \leq b^G_*(Y;D)$.

Universality allows us to remove open stars of edges of $L$ with $L^2$-acyclic links rather than vertices and still have some control over $b_\ast^{(2)}$.
Removing edges is somewhat unnatural, as the remaining complex $R = L \smallsetminus \mathring{\St}_Le$ is not a full subcomplex of $L$, hence the homomorphism $W_{R} \rightarrow W_{L}$ induced by inclusion is not injective and the lifts of $P_R$ to $\Sigma_L$ are not copies of $\Sigma_R$.
However, we still get a decomposition of $\Sigma_L$, so can use Mayer--Vietoris, and the resulting lift of $P_R$ in $\Sigma_L$ is precisely the covering space associated to the homomorphism $W_{R} \rightarrow W_{L}$.
Therefore, we get an intermediate cover of $P_R$ which has the same $L^2$-Betti numbers as $\Sigma_L$.

Since the commutator subgroups $C_L, C_R$ of $W_L$ and $W_R$ respectively are special in Haglund and Wise's sense, they are residually torsion-free nilpotent, and hence the $L^2$-Betti numbers of both the universal cover and the intermediate cover can be computed using the skew fields $\cD_{\Q C_R}$ and $\cD_{\Q C_L}$ respectively.
So, universality implies that the $L^2$-Betti numbers of the universal cover $\Sigma_R$ are bounded above by the $L^2$-Betti numbers of the intermediate cover; hence if $\Sigma_L$ has vanishing $L^2$-homology so does $\Sigma_R$.

For example, if $L$ is the $1$-skeleton of a $3$-cube, then $L$ is a subcomplex of the suspension of a hexagon, and is obtained from this suspension by removing $6$ edges.
The link of each of these edges is two points, hence $\Sigma_{\Lk e} \cong \rr$ which is $L^2$-acyclic.
Since the Davis complex of the suspension of the hexagon is $L^2$-acyclic and removing one edge does not change the links of the others, the $L^{2}$-acyclicity of $\Sigma_L$ follows.

One can also phrase these edge removals in terms of edge subdivisions, and it is usually more convenient to think in these terms.
As the Davis complex associated to the link of a midpoint of a subdivided edge is $L^2$-acyclic, Mayer--Vietoris implies that the Davis complexes of the subdivided complex and $L \smallsetminus \oSt e$ have the same $L^2$-Betti numbers.
We use the fact that $b\d \Delta^n$ is an iterated edge subdivision of the boundary of the $n$-octahedron to complete the proof of Theorem \ref{t:main} (\ref{t:mainitemone}), and using part (\ref{t:mainitemone}), part (\ref{t:mainitemtwo}) follows from similar Mayer--Vietoris arguments as in \cite{do04}.
More generally, Volodin in \cite{v10} showed that if $L$ is a flag nested set complex associated to a building set for $\d \Delta^n$ (see e.g.
\cite{pnw08}), then $L$ is an iterated edge subdivision of the boundary of an $n$-octahedron.
Our argument proves the Singer conjecture for such $P_L$.

\subsection{Torsion growth} In Theorem \ref{t:vanishingtomei} and Corollary \ref{c:vanishingbs}, we prove a more general version of Theorem \ref{t:main}.
Firstly, we prove vanishing of homology with $\DFG$-coefficients for general $\F$ rather than just the usual $L^2$-Betti numbers.
We also consider relative barycentric subdivisions, where we only subdivide simplices outside of some flag subcomplex.
One motivation for considering these two generalizations comes from a variation of the Singer conjecture for torsion homology growth for closed aspherical manifolds.
\begin{torsionconj}[\cite{l13}*{1.12(2)}]\label{luckconj}
    Let $M^n$ be a closed aspherical $n$-manifold with residually finite fundamental group.
    Let $G_k \nsgp G = \pi_1(M^n)$ be any normal chain of finite index subgroups with $\bigcap_{k} G_k = 1$.
    If $i \neq (n-1)/2$, then
    \[
    \limsup_k \frac{\logtor H_{i}(G_k;\zz)} {[G:G_k]} = 0.
    \]
\end{torsionconj}
L\"{u}ck's Approximation Theorem identifies $b_\ast^{(2)}(\widetilde M_n; G)$ with the rational homology growth of the $G_k$.
Even when $\cD_{\F_p G}$ is known to exist, there is no complete statement relating $b_\ast^{G}(\widetilde M^n; \cD_{\F_p G})$ with the $\F_p$-homology growth, however there are approximation results for certain chains.
Now, in \cite{aos21}, we showed that there existed flag triangulations of $S^6$ such that the corresponding manifolds $P_{S^{6}}$ had linear $\F_p$-homology growth outside of the middle dimension (for some chains).
If these manifolds additionally satisfied the Singer conjecture, then universal coefficients would imply these manifolds were counterexamples to the Torsion concentration conjecture.
However, we could not prove the Singer conjecture for these examples, hence could only say the two conjectures were incompatible.

Our construction in~\cite{aos21} started from a flag triangulation $T$ of $S^6$ which was a barycentric subdivision relative to an embedded complex $OL$ with $b_4^{(2)}(W_{OL}; \Fp) \ne 0$.
We showed that the $\Fp$-version of the Singer conjecture failed either for the right-angled Coxeter group associated to $T$ or to the link of an odd-dimensional simplex in $T$.
In either case, this produced a $7$-dimensional counterexample, though in a non-explicit way.

Using the computation in Theorem \ref{t:vanishingtomei}, we can improve these counterexamples in two ways.
We can show now that $T$ itself is a counterexample to the $\Fp$-version of the Singer conjecture.
Moreover, we also produce flag triangulations of $S^{13}$, which satisfy the Singer conjecture but violate the $\Fp$-version.
We can promote these examples to higher even dimensions.
These appear to be the first examples of closed aspherical manifolds of this sort.
Although we still do not know of any such manifolds in odd dimensions, we obtain explicit counterexamples to the Torsion concentration conjecture in even dimensions $\geq 14$ and all dimensions $\geq 21$.
We note that in \cite{kkl23}*{Conjecture 6.11}, a modified version of the conjecture has been proposed, which is still open.
Instead of concentration, it predicts that the alternating sum of torsion growths of an odd-dimensional closed aspherical manifold limits to the $L^2$-torsion of the universal cover.

\section{Skew fields}\label{s:skew}

We now recall the skew field approach to $L^2$-Betti numbers and their finite field analogues.
Let $G$ be a group, $\F$ a field, $Y$ be a free $G$-CW complex, and suppose we have a homomorphism $\phi: \F G \to D$ to a skew field.
The composition of $\phi$ with the homomorphism $\zz G \to \F G$ turns $D$ into a $\zz G$-bimodule, so we can take equivariant homology of $Y$ with coefficients in $D$
\[
H_{*}^{G}(Y ; D)=H_{*}( D \tens_{\zz G} C_*(Y) ),
\]
and define the \emph{equivariant Betti numbers with coefficients in $D$ of $Y$} by taking its dimension over $D$:
\[
b^G_{*}(Y; D)= \dim_{D} H_{*}^{G}( Y; D ).
\]

We record some easy basic facts about $H_{*}^{G}( Y; D )$.
\begin{lemma}\label{l:agra}
    Suppose that $Y$ is a free $G$-complex and $\phi: \F G \rightarrow D$ where $D$ is a skew field.
    Then we have the following:
    \begin{enumerate1}

        \item Suppose $Y_0 \subset Y$ is a $G$-invariant subcomplex, then there is a long exact sequence
        \[
        \dots \to H^G_k(Y_0;D)\to H^G_k(Y;D) \to H^G_k(Y, Y_0; D)\to H^G_{k-1}(Y, Y_0; D) \to \dots
        \]
        \item If $Y=Y_1 \cup Y_2$ where $Y_1$ and $Y_2$ are $G$-invariant subcomplexes.
        Then $Y_1 \cap Y_2$ is $G$-invariant, and we have a Mayer--Vietoris sequence
        \[
        \dots \to H^G_k(Y_1 \cap Y_2;D)\to H^G_k(Y_1;D)\oplus H^G_k(Y_2;D)\to H^G_k(Y;D)\to H^G_{k-1}(Y_1 \cap Y_2;D) \to \dots
        \]
        \item\label{i:contain}
        If $D$ is contained in another skew field $D'$, then
        \[
        b^{G}_{*}(Y;D)=b^{G}_{*}(Y;D'),
        \]
        where the latter is computed using the composition $\F G \to D \into D'$.

        \item\label{i:induce}
        If $H < G$ is a subgroup and $Y_{0}$ is a free $H$-complex, then for $Y=G \times_{H} Y_{0}$ we have
        \[
        b^{G}_{*}(Y;D) = b^{H}_*(Y_{0};D),
        \]
        where the latter is computed using the composition $\F H \into \F G \to D$.
    \end{enumerate1}
\end{lemma}

If $H \nsgp G$ and $\phi: \F H \rightarrow D$, let $\phig : \F H \to \Dg $ denote the composition of $\phi$ with the left conjugation by $g$, $\phig(h)=\phi(gh g^{-1} )$, so that $\Dg = D$ as abstract skew fields, but $\Dg \neq D$ as $\F H$-skew fields.
We shall need the following generalization of Lemma \ref{l:agra}\eqref{i:induce}.
\begin{lemma}\label{l:mess}

    Suppose $H \nsgp G$, $K<G$, $Y$ is a free $K$-complex and $\phi: \F H \rightarrow D$ where $D$ is a skew field.
    Let $P=HK < G$, then
    \[
    H^{H}_{*}( P\times_{K} Y;D) = H^{H\cap K}_*(Y;D)
    \]
    where the right-hand side is computed using the composition $\F [H \cap K] \into \F H \to D$, and
    \[
    H^{H}_{*}(G \times_{K} Y;D) = \bigoplus_{gP \in G/P} H^{H\cap K}_*(Y;\Dg)
    \]
    where the sum is over cosets $G/P$, and $\phig : \F H \to \Dg $ is the composition of $\phi$ with the left conjugation by $g$, $\phig(h)=\phi(gh g^{-1} )$ and the right-hand side uses $\F [H \cap K] \into \F H \to \Dg$.
    Moreover, the terms on the right-hand side are independent of the choice of representatives of cosets.
\end{lemma}
\begin{proof}
    The map $\zz P \rightarrow \zz H \tens_{\zz [H \cap K]} \zz K$ sending $hk \to h \tens k$ is well-defined, and it induces an isomorphism of $H$-$K$-modules.
    Then both sides in the first formula are homology of the same complex:
    \[
    D\tens_{\zz H} C_{*}( P\times_{K} Y)= D \tens_{\zz H} \zz P \tens_{\zz K} C_{*} (Y) \cong D \tens_{\zz H} \zz H \tens_{\zz [H \cap K]} \zz K \tens_{\zz K} C_{*}(Y) = D \tens_{\zz [H \cap K]} C_{*}(Y).
    \]

    For the second formula, we have a direct sum decomposition $\zz G= \bigoplus_{[g ] \in G/P} \zz[gP]$ as $H$-$K$-modules, since $h(gp)k= gg^{-1}hgpk=g (h^{g}pk)$.

    The map $d \tens p \mapsto d \tens gp$ induces an isomorphism of $D$-$K$-modules $\Dg\tens_{\zz H} \zz P \cong D\tens_{\zz H} \zz[gP] $, because it maps $hp$ to $ghp= ghg^{-1}gp= \leftindex[T]^{g}h gp$, and hence the tensor relation for $h$ in $\Dg\tens_{\zz H} \zz P$,
    \[
    d \leftindex[T ]^{g}h \tens p = d \tens hp,
    \]
    maps to the tensor relation for $\leftindex[T]^{g}h$ in $D\tens_{\zz H} \zz[gP]$:
    \[
    d \leftindex[T ]^{g}h \tens gp= d \tens\leftindex[T]^{g}h gp.
    \]
    Note that this implies that $\Dg\tens_{\zz H} \zz P$ is independent of the choice of representative $g$ of $gP$.
    Thus we have isomorphisms of chain complexes:
    \[
    \begin{split}
        D\tens_{\zz H} C_{*}( G \times_{K} Y)= D\tens_{\zz H} \zz G \tens_{\zz K} C_{*} (Y)= \bigoplus_{gP \in G/P} D\tens_{\zz H} \zz[gP] \tens_{\zz K} C_{*} (Y) \cong \\
        \bigoplus_{gP \in G/P} \Dg \tens_{\zz H} \zz P \tens_{\zz K} C_{*} (Y) = \bigoplus_{gP \in G/P} \Dg \tens_{\zz H} C_{*}( P\times_{K} Y)
    \end{split}
    \]
    Taking homology and noting that by the first part homology of each summand is $H^{H\cap K}_*(Y;\Dg)$ proves the second formula.
\end{proof}

A rather unpleasant feature of the general theory is that in general, $\Dg$ and $D$-Betti numbers are different.
If we put additional strong algebraic assumptions on $G$, then there is essentially a canonical choice of $D$, which avoids this problem, and which we now explain.

\subsection{Linnell skew fields}

Recall that a homomorphism to a skew field $D$ is \emph{epic} if the division closure of its image is $D$.
\begin{definition}
    We say that an epic skew field $\phi: \F G \rightarrow D$ is \emph{Linnell} if for any subgroup $K < G$, the map
    \[
    D_{|K} \tens_{\F K} \F G \rightarrow D; \qquad d \tens f \rightarrow d\phi(f)
    \]
    is injective, where $D_{|K}$ denotes the division closure of $\phi(\F K)$ in $D$.
\end{definition}

Note that if $D$ is a Linnell skew field, then $\phi: \F G \rightarrow D$ is injective, so at the very least $G$ must be torsion-free.
A fundamental theorem of Hughes \cite{h70} implies that for locally indicable groups Linnell skew fields are unique up to $\F G$-isomorphism.
We will say \emph{$\DFG$ exists} to mean that $G$ is locally indicable and there is a Linnell skew field $\phi: \F G \to \DFG$.
This condition passes to subgroups: if $K<G$ and $\DFG$ exists, then ${\DFG}_{|K}$ is Linnell over $\F K$, hence $\cD_{\F K}={\DFG}_{|K}$ exists.

Suppose $H < G$ is a finite index subgroup and $\DFG$ exists.
Then the Linnell condition implies that $\DFG \cong \DFH^{[G:H]}$ as vector spaces over $\DFH$, and therefore the $\DFG$-Betti numbers satisfy the same multiplicativity as the usual $L^2$-Betti numbers:
\[
b^{H}_{*}(Y; \DFH)=[G:H] b^{G}_{*}(Y; \DFG).
\]

This multiplicativity allows us to extend skew field Betti numbers to a virtual setting.
We will say \emph{$\DFG$ virtually exists} if there is a finite index subgroup $H < G$ such that $\DFH$ exists.
If $\DFG$ virtually exists and $Y$ is a proper $G$-complex, we define $b_\ast^{(2)}(Y; \F) := \frac{1}{[G:H]}b_\ast^{H}(Y; \DFH)$ for a finite index subgroup $H$ such that $\DFH$ exists, and multiplicativity implies this does not depend on the choice of $H$ (but they do depend on $G$, although we omit it from the notation).
We will refer to $b_\ast^{(2)}(Y; \F) $ as the $\F$-$L^{2}$-Betti numbers of $Y$ with respect to the $G$-action.
We will denote $b_\ast^{(2)}(G; \F) := b_\ast^{(2)}(Y; \F) $ for a contractible $Y$.
\begin{lemma}\label{l:ind}
    Suppose $ K< G$ is a subgroup, $\DFG$ virtually exists, and $Y$ is a proper $K$-complex.
    Then
    \[
    b^{(2)}_{*}(G\times_{K}Y; \F) = b^{(2)}_*(Y; \F).
    \]
\end{lemma}
\begin{proof}
    Let $H \nsgp G$ be a finite index normal subgroup such that $\DFH$ exists.
    Then by Hughes' theorem $\DFH^{g} \cong \DFH$ as $\F H$-rings.
    So by Lemma~\ref{l:mess}
    \[
    \begin{split}
        b^{(2)}_{*}(G\times_{K}Y; \F) = \frac{b_\ast^{H}(G\times_{K}Y; \DFH)}{[G:H]}= \frac{[G:HK] b_\ast^{H\cap K }(Y; \DFH)}{[G:H]} = \frac{[G:HK] b_\ast^{H\cap K }(Y; \cD_{\F[H\cap K]})} {[G:H]} \\
        = \frac{[G:HK] [K:H\cap K]b^{(2)}_*(Y; \F)}{[G:H]} = \frac{[G:HK] [HK: H] b^{(2)}_*(Y; \F)}{[G:H]} = b^{(2)}_*(Y; \F).
    \end{split}
    \]
\end{proof}

If $\F = \Q$ and $G$ is locally indicable, then Jaikin-Zapirain and López-Álvarez~\cite{jl20} showed that $\DQG$ exists, and if $Y$ is a free cocompact $G$-complex, then $b^{G}_{*}(Y; \DQG)$ agrees with the usual $L^2$-Betti numbers $b^{(2)}_{*}(Y; G)$.

Moreover, for general $\F$, Jaikin-Zapirain in \cite{j21} showed that if $G$ is residually (locally indicable amenable), then $\DFG$ exists and is universal in the sense that for any other $\F G$-skew field $D$, $b^{G}_{*}(Y; \DFG) \leq b^{G}_{*}(Y; D)$.

We now explain a somewhat simpler argument for a weaker form of universality for residually torsion-free nilpotent groups, which is sufficient for our purposes.
Note that a residually torsion-free nilpotent group $G$ is bi-orderable, and in this case Malcev and Neumann showed that the set of functions $f: G \rightarrow \F$ with well-ordered support endowed with convolution multiplication is a skew field $MN_{<}(G)$.
We can take $\DFG$ to be the division closure of $\F G$ inside a (any) $MN_{<}(G)$.

\subsection{Relation to homology growth} In \cite{aos24}, we used work of Linnell, L\"uck, and Sauer in \cite{lls11} to prove the following:
\begin{Theorem}[\cite{aos24}*{Theorem 3.6}]\label{t:binf}
    Let $G$ be a residually torsion-free nilpotent group and let $Y$ be a free cocompact $G$-complex.
    Then

    \[
    b^{G}_{*}(Y; \DFG) = \inf_{\substack{H <G \\
    [G:H] < \infty}} \frac{b_{*}(Y/H; \F)}{[G:H]}.
    \]
\end{Theorem}

An immediate corollary is a weak form of universality of $\DFG$.
\begin{lemma}\label{l:universality}
    Suppose $K \nsgp G \onto Q$, where $G$ and $Q$ are residually torsion-free nilpotent.
    Let $Y$ be a free cocompact $G$-complex.
    Then
    \[
    b^{G}_{*}(Y; \DFG) \le b^{Q}_{*}(Y/K ; \DQ).
    \]
\end{lemma}
(Note that $b^{Q}_{*}(Y/K ; \DQ)= b^{G}_{*}(Y; \DQ) $ where the later is computed using $\F G \to \F Q \to \DQ$, so the inequality is a special case of universality.)
\begin{proof}
    By Theorem \ref{t:binf},
    \[
    b^{Q}_{*}(Y/K ; \DQ) = \inf_{\substack{ K < H <G \\
    [G:H] < \infty}} \frac{b_{*}(Y/H; \F)}{[G:H]} \geq \inf_{\substack{H <G \\
    [G:H] < \infty}} \frac{b_{*}(Y/H; \F)}{[G:H]} = b^{G}_{*}(Y; \DFG).
    \]
\end{proof}

In fact, for $\F=\Fp$ much smaller collections of subgroups suffice for the conclusion of Theorem~\ref{t:binf}.
A version of the following theorem (phrased in terms of residual chains of subgroups) was explained to us by Jaikin-Zapirain.
Let $\mathcal{P}(G)$ be the poset of finite index normal subgroups $H < G$ with $G/H$ a $p$-group, ordered by inverse inclusion.
\begin{Theorem}\label{t:bpinf}
    Let $G$ be a residually torsion-free nilpotent group and let $Y$ be a free cocompact $G$-complex.
    Then for any cofinal subset $\mathcal{Q}$ of $\mathcal{P}(G)$,
    \[
    b^{G}_{\ast}(Y; \cD_{\F_pG}) = \inf_{H \in \mathcal{Q}} \frac{b_{\ast}(Y/H; \F_p)}{[G:H]}.
    \]
\end{Theorem}
\begin{proof}
    In \cite{aos24} we first used the fact that $b^{G}_{*}(Y; \cD_{\F_pG}) = b^{N}_{*}(Y/K; \cD_{\F_pN})$ for some torsion-free nilpotent quotient $K \rightarrow G \rightarrow N$.
    Then the main theorem of \cite{lls11} implies
    \[
    b^{N}_{*}(Y/K; \cD_{\F_pN}) = \lim_{i \rightarrow \infty} \frac{b_\ast(Y/H_i; \F_p)}{[G:H_i]}
    \]
    for any nested sequence $H_i$ of finite index normal subgroups with $\medcap H_i = K$.
    Since torsion-free nilpotent groups are residually $p$-finite for any $p$, we can assume that $H_i \in \mathcal{P}(G)$.
    This implies the claim in the case $\mathcal{Q}=\mathcal{P}(G)$, since by Theorem~\ref{t:binf}, $b^{G}_{*}(Y; \cD_{\F_p G}) \le \frac{b_\ast(Y/H; \F_p)}{[G:H]}$ for any $H \in \mathcal{P}(G)$.

    Now, \cite{blls14}*{Theorem 1.6} implies that if $H' < H$ and $H, H' \in \mathcal{P}(G)$, then
    \[
    \frac{b_\ast(Y/H'; \F_p)}{[G:H']} \le \frac{b_\ast(Y/H; \F_p)}{[G:H]}.
    \]
    This monotonicity implies that restricting to a cofinal collection does not change the $\inf$.
\end{proof}

A useful corollary of Theorem \ref{t:bpinf} is a K\"unneth formula for $\F$-$L^{2}$-Betti numbers, at least for virtually residually torsion-free nilpotent $G$.
\begin{kunneth}\label{l:kunneth}
    Suppose $Y_i$ is a proper cocompact $G_i$-complex for $i = 1,2$ with $G_i$ virtually residually torsion-free nilpotent.
    Then
    \[
    b^{(2)}_n(Y_1 \times Y_2; \F) = \sum_{i + j = n} b_i^{(2)}(Y_1; \F ) b_j^{(2)}(Y_2,\F) .
    \]
\end{kunneth}
\begin{proof}
    For $L^2$-Betti numbers, the K\"unneth formula holds generally, see \cite{l02}*{Theorem 1.3.5}, so suppose $\F=\Fp$.
    Since the index is multiplicative for direct products, we can assume that $G_{i}$ are residually torsion-free nilpotent.
    Then, since $\mathcal{P}(G_1) \times \mathcal{P}(G_2)$ is cofinal in $\mathcal{P}(G_1 \times G_2)$, the usual K\"unneth formula for appropriate $Y_1/H_1 \times Y_2/H_2$ and Theorem \ref{t:bpinf} imply the formula.
\end{proof}

\section{Right-angled Coxeter groups}\label{s:racg}

Suppose $L$ is a flag complex with vertex set $S$.
Associated to $L$ there is a right-angled Coxeter group (abbreviated RACG), $W_L$, defined by a presentation as follows.
A set of generators for $W_L$ is $S$; and there are relations $s^2=1$ for all $s\in S$, and $[s, t] = 1$ whenever $\{s,t\}$ span an edge in $L$.
We say that $L$ is the nerve of $W_L$.

We will let $C_L$ denote the commutator subgroup $[W_L, W_L]$.
It is the kernel of the obvious homomorphism $W_L \rightarrow (\zz/2)^{\abs{L^{0} }}$; thus it is finite index in $W_L$.
It also happens to be torsion-free, and in fact a subgroup of the corresponding right-angled Artin group $A_L$.
Indeed, given $W_L$, Davis and Januszkiewicz \cite{dj00} showed that $W_L$ is a subgroup of a larger right-angled Coxeter group $W_{\Delta L}$ which has $A_L$ as a subgroup of index $2^{\abs{L}}$, and in fact, $C_L = W_L \cap A_L$ inside of $W_{\Delta L}$ \cite{los12}*{Proposition 5}.
Right-angled Artin groups and their subgroups are residually torsion-free nilpotent, so the results in Section \ref{s:skew} apply to virtually special groups, and in particular to RACG's.
\begin{definition}
    A \emph{mirror structure} on a CW-complex $X$ consists of an index set $S$ and a family of subcomplexes $\{X_s\}_{s \in S}$.
    The subcomplexes $X_s$ are the \emph{mirrors} of $X$.
    We say $X$ is a mirrored space over $S$.
\end{definition}

Let $X$ be a mirrored space over $S$, and let $W$ be a RACG with vertex set $S$.
Given $x$, let $S(x) := \{s \in S\mid x \in X_s\}$.
Given any Coxeter group generated by $S$, let $\cu(W, X )$ be the \emph{basic construction} associated to this data:
\[
\cu(W, X) := (W \times X) /\sim,
\]
where $(w,x) \sim (w',x')$ if and only if $x = x'$ and $wW_{S(x)} = w'W_{S(x)}$.

For right-angled Coxeter groups there is a natural choice of a mirrored space $X$.
Given a flag complex $L$, let $K_L$ be the geometric realization of the poset of simplices of $L$.
Then $K_L$ is isomorphic to the cone on the barycentric subdivision of $L$, with the empty simplex corresponding to the cone point.
There is a canonical mirror structure on $K_L$ where the $s$-mirror $K_s$ is the geometric realization of the subposet of simplices containing the vertex $s$.
This is isomorphic to the star of $s$ in the barycentric subdivision of $L$.
The \emph{Davis complex} $\Sigma_L$ is the basic construction $\cu(W_L, K_L)$.

The Davis complex admits a natural cubical structure.
Firstly, $K_L$ can be naturally identified with a CAT(0) cubical subcomplex of $[0,1]^{S}$; $cL$ is precisely the union of subcubes of $[0,1]^{S}$ containing $(0,0, \dots, 0)$ corresponding to the collections of vertices which span simplices in $L$.
Note that the link of the vertex $(0,0,\dots,0)$ is isomorphic to $L$, and the vertices of each cube can be identified with vertices of the barycentric subdivision of $L$.

For a vertex $s$ the mirror $K_{s}$ is the intersection of $K_L$ with the hyperplane $x_{s}=1$.
It is naturally isomorphic to the cubical complex $K_{\Lk s}$.
This cubical structure naturally extends to $\Sigma_L$.
\begin{figure}
    \centering
    \begin{tikzpicture}[scale = .7]
        \begin{scope}[xshift = -3cm, yshift = 0cm, scale = .8]
            \draw[fill=lightgray, opacity = .5] (0,0)--(4,0)--(4,4)--(2,4)--(2,2)--(0,2)--(0,0);
            \draw[thick] (0,0)--(4,0)--(4,4)--(2,4)--(2,2)--(0,2)--(0,0);
            \draw[thick] (2,0) -- (2,2) -- (4,2);
            \draw[dashed] (0,0) -- (0,4) -- (2,4);
            \node[above] at (3,4) {$K_s$};
            \node[left] at (0,1) {$K_t$};
            \node[below] at (2,0) {$K_u$};
            \node[right] at (4,2) {$K_v$};

            \node[below] at (2,-2) {$K_{L- e} \subset K_L$};
        \end{scope}
        \begin{scope}[xshift = -7cm, yshift = 1.6 cm, scale = 1.6]

            \draw[thick] (-1,0) coordinate (t) -- (0,-1) coordinate (u)-- (1,0) coordinate (v)-- (0,1) coordinate (s);

            \draw[dashed] (s) -- node[midway, above left] {$e$} (t);

            \node[below] at (0,-2) {$L- e \subset L$};
            \node[above] at (s) {$s$};
            \node[left] at (t) {$t$};
            \node[below] at (u) {$u$};
            \node[right] at (v) {$v$};
        \end{scope}
        \begin{scope}[xshift = 6cm, yshift = 2cm, scale = .7]

            \draw[fill=lightgray, opacity = .5] (-1,-2)--(0,-2)--(0,-4)--(1,-4)--(1,4)--(-1,4)--(-1,2)-- (0,2) -- (0,0) -- (-1,0) -- (-1,-2);
            \draw[fill=lightgray, opacity = .5] (-2,-4) -- (-2,4) -- (-3,4) -- (-3,-4) -- (-2,-4);
            \draw[fill=lightgray, opacity = .5] (-1,4) -- (-2,4) -- (-2,2) -- (-1,2) -- (-1,4);
            \draw[fill=lightgray, opacity = .5] (-1, 0) -- (-2,0) -- (-2,-2) -- (-1,-2) -- (-1,0);
            \draw[fill=lightgray, opacity = .5] (4, 4) -- (1,4) -- (1,2) -- (4,2) -- (4,4);
            \draw[fill=lightgray, opacity = .5] (1, 2) -- (2,2) -- (2,-4) -- (1,-4) -- (1,2);
            \draw[fill=lightgray, opacity = .5] (2, 0) -- (4,0) -- (4,-2) -- (2,-2) -- (2,0);
            \draw[fill=lightgray, opacity = .5] (4, 4) -- (5,4) -- (5,-4) -- (4,-4) -- (4,4);
            \draw[fill=lightgray, opacity = .5] (-3, -5) -- (5,-5) -- (5,-4) -- (-3,-4) -- (-3,-5);
            \draw[thick] (1,1) -- (1, -1) -- (-1,-1) -- (-1,0) -- (-0,0) -- (0,1) -- (1,1);
            \draw[step=1cm] (-3,-5) grid (5,4);

            \node[right] at (5.5,0) {$\dots $};
            \node[left] at (-3.5,0) {$\dots $};
            \node[left] at (-3.5,4.5) {$\ddots $};
            \node[] at (5.5,4.5) {$\Ddots $};
            \node[right] at (5.5,-5.5) {$\ddots $};

            \node[above] at (1,4.5) {$\vdots $};
            \node[below] at (1,-5.5) {$\vdots $};
            \node[] at (-3.5,-5.5) {$\Ddots $};

            \node at (1,-8) {$W_L K_{L - e} \subset \Sigma_L \cong \rr^2$};
        \end{scope}
    \end{tikzpicture}

    \caption{If $L$ is a $4$-cycle, then $W_L \cong D_\infty \times D_\infty$ and $\Sigma_L$ is the standardly cellulated $\rr^2$.
    Removing an edge from $L$ corresponds to removing a square from $K_L$.
    Reflecting this around gives $\rr^2$ with a $W_L$-orbit of squares deleted.}
\end{figure}
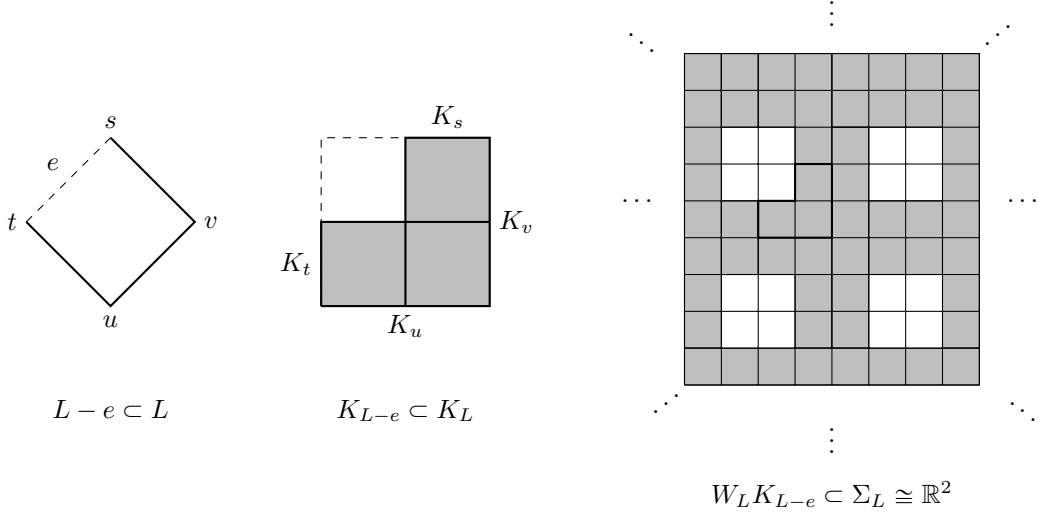

\subsection{Flag subcomplexes}

Given a flag subcomplex $A$ of $L$, there are canonical inclusions $K_{A} \into K_{L}$ and $P_{A} \into P_{L}$ and a homomorphism $\phi_{AL}: W_{A} \to W_{L}$ induced by the inclusion of the generating set of $A$ into the generating set of $W_{L}$.
The inclusion $K_{A} \into K_{L}$ and $\phi_{AL}$ induce a natural map $\gS_A \to \gS_L$.
Since $\phi_{AL}$ induces an injective map on abelianizations $(\zz/2)^{\abs{A^{0}}} \rightarrow (\zz/2)^{\abs{L^{0}}}$, the kernel of $\phi$ is contained in $C_{A}$.
Therefore, it restricts to a homomorphism $C_{A} \to C_{L}$, and $\gS_A \to \gS_L$ is the map of universal covers induced by the inclusion $P_{A} \into P_{L}$.
Note that the image of $\gS_A$ in $ \gS_L$ is $\gS_A/\ker \phi_{AL}$, and the union of $W_{L}$-translates of this image is the complex $W_{L}K_{A}$.
We now describe the relationship between $\gS_{A}$ and this complex.

Let $\bar A$ denote the full subcomplex of $L$ generated by the vertices of $A$.
The inclusion $A \into L$ factors through $\bar A$, $A \into \bar A \into L$, and we have $\phi_{A L}=\phi_{\bar A L}\circ \phi_{A \bar A}$.
Since $\bar A$ is a full subcomplex of $L$, $\phi_{\bar A L}$ is injective and $\gS_{\bar A} \to \gS_L$ is an embedding.
If $w \in W_L - W_{\bar A}$, then $\Sigma_{\bar A}$ and $w\Sigma_{\bar A}$ are disjoint; hence the stabilizer of $\Sigma_{\bar A}$ is $W_{\bar A}$ so we get
\[
W_{L} K_{\bar A} =W_{L}\times_{W_{\bar A}} \gS_{\bar A}/\ker \phi_{\bar A L} .
\]
Now, since $A^{0}=\bar A^{0}$, the map $\phi_{A \bar A} $ is onto.
Hence, the image of $\gS_{A}$ in $\gS_{\bar A}$ is $W_{\bar A}$-invariant, and therefore its stabilizer in $W_{L}$ is still $W_{\bar A}$, and elements of $W_L - W_{\bar A}$ move it off itself.
Therefore, since $\ker \phi_{A \bar A}= \ker \phi_{A L}$, we have
\[
W_{L} K_A =W_{L} \gS_{A}/\ker \phi_{A L} =W_{L}\times_{W_{\bar A}} \gS_{A}/\ker \phi_{A L} .
\]
Note that, since $\phi_{A \bar A} $ restricts to an onto map $C_{A} \to C_{\bar A}$, $W_{\bar A} K_A $ is an intermediate cover of $P_{A}$.

As a simple example consider the inclusion of a flag $A$ into the simplex $\gD$ on the same vertices.
Then $\ker \phi_{A \gD}=C_{A}$ and $W_{\gD} K_A =P_{A}$ is a subcomplex of the cube $\gS_{\gD}$.
\begin{notation}
    If $A$ is just a subset of simplices of $L$, we denote by $L \smallsetminus A$ the set of simplices of $L$ not contained in $A$, and we denote by $L - A $ the set of simplices which do not contain any simplex of $A$.
    In particular, if $A$ is a subcomplex of $L$, then $L - A$ is the full subcomplex of $L$ spanned by vertices in $L \smallsetminus A$.
\end{notation}

\subsection{Mayer--Vietoris arguments}

We now describe the effect of vertex and edge removal on $\F$-$L^{2}$-Betti numbers of the Davis complex $\gS_{L}$ with respect to the $W_{L}$-action.
By abusing terminology, we will sometimes say a flag complex $L$ is $\F$-$L^{2}$-acyclic if $b_*^{(2)}(W_{L}; \F)=0$.
\begin{lemma}
    [cf.
    \cite{aos21}*{Lemma 11}]\label{l:v-remove}
    If $L$ is a flag complex and $v \in L$ is vertex with $b_{i}^{(2)}(W_{\Lk v}; \F) =0$, then
    \[
    b_i^{(2)}(W_{L-v}; \F) \leq b_i^{(2)}(W_{L}; \F),
    \]
    and, if in addition, $ b_{i-1}^{(2)}(W_{\Lk v}; \F) = 0$, then
    \[
    b_i^{(2)}(W_{L-v}; \F) = b_i^{(2)}(W_{L}; \F).
    \]
\end{lemma}
\begin{proof}
    If $v$ is a vertex of a flag complex $L$, then the star $\St v$, the link $\Lk v$, and $L - v = L \smallsetminus \oSt v$ are all full subcomplexes of $L$.
    The decomposition
    \[
    L= \St_{v} \medcup_{\Lk v} ( L-v)
    \]
    induces a decomposition
    \[
    \Sigma_L = W_L\Sigma_{\St v} \medcup_{W_L\Sigma_{\Lk v}} W_L\Sigma_{L-v},
    \]

    Since all subcomplexes are full, we have a decomposition:
    \[
    \Sigma_L = W_L \times_{W_{\St v}} \Sigma_{\St v} \medcup_{W_L\times_{W_{\Lk v}} \Sigma_{\Lk v}} W_L \times_{W_{L-v}} \Sigma_{L-v},
    \]
    hence a Mayer--Vietoris sequence with $\DCL$-coefficients:
    \[
    \dots \to H_{i+1}^{C_{L}}(W_L\times_{W_{\Lk v}} \gS_{\Lk v}) \to H_{i}^{C_{L}}(W_L \times_{W_{\St v}} \gS_{\St v}) \oplus H_{i}^{C_{L}}(W_L \times_{W_{L-v}} \Sigma_{L-v}) \to H_{i}^{C_{L}}(\gS_{L}) \to \dots
    \]
    Since the star is the cone on the link, its Davis complex and the associated RACG split as products
    \[
    \Sigma_{\St v} \cong \Sigma_{\Lk v} \times [0,1], \qquad W_{\St v} \cong W_{\Lk v} \times \zz/2.
    \]
    Hence $b_i^{(2)}(W_{\St v}; \F) = \frac{1}{2} b_i^{(2)}(W_{\Lk v}; \F)$, and the assumptions and Lemma~\ref{l:ind} imply vanishing of the corresponding terms in the sequence, which implies the inequalities.
\end{proof}

The corresponding statement for the removal of edges is our main technical result in this paper.
Here, we treat edges as single simplices of $L$, not subcomplexes.
\begin{Theorem}\label{t:subdivideedge}
    Suppose we have a flag complex $L$ and an edge $e$ whose link $\Lk e$ has $b_{i-1}^{(2)}(W_{\Lk e}; \F) = 0$ for some $i$.
    Then
    \[
    b_i^{(2)}(W_{L - e}; \F) \le b_i^{(2)}(W_{L}; \F).
    \]
    In particular, if $b_i^{(2)}(W_{L}; \F) = 0$ then
    \[
    b_i^{(2)}(W_{L - e}; \F) = 0.
    \]
\end{Theorem}
\begin{proof}
    Let $R = L - e$.
    Then we have $R = L \smallsetminus \oSt e$, and a decomposition
    \[
    L = \St e\medcup_{S\Lk e} R
    \]
    where the star $\St e$ is a full subcomplex, but the suspension of the link and complement are flag \emph{non-full} subcomplexes of $L$.
    This gives a decomposition $K_L = K_{\St e} \cup_{K_{S\Lk e}} K_{R}$, and hence a decomposition of $\Sigma_L$:
    \[
    \Sigma_L = W_{L}\Sigma_{\St e} \textstyle \bigcup\limits_{W_L K_{S\Lk e}} W_L K_{R}.
    \]
    Consider the long exact sequence of the pair $(\gS_{L}, W_L K_R ) = W_{L}(K_{L}, K_{R}) $.
    \[
    \dots \to H_{i+1}^{C_{L}}(W_{L} (K_{L},K_{R}); \DCL) \to H_{i}^{C_{L}}(W_{L}K_{R}; \DCL) \to H_{i}^{C_{L}}(\gS_{L} ; \DCL) \to \dots
    \]
    The pair $W_{L}(K_{L}, K_{R})$ excises to the pair $W_{L}(K_{\St e}, K_{S\Lk e})$, and the key observation is that since the star $\St e$ is full, this pair is just $W_L \times_{W_{\St e} }(\gS_{\St e}, W_{\St e} K_{S\Lk e} )$.
    Next, the splitting $\St e = \Lk e * e$ gives splittings
    \[
    K_{\St e} =K_{\Lk e} \times K_{e}, \quad K_{S\Lk e} = K_{\Lk e} \times K_{\d e}, \quad W_{\St e} =W_{\Lk e} \times W_{e}.
    \]
    Therefore, $(\Sigma_{\St e}, W_{\St e} K_{S\Lk e}) \cong \Sigma_{\Lk e} \times W_{e}(K_{e}, K_{\d e}) $.
    Now, since $W_{e}(K_{e}, K_{\d e}) \cong ([0,1]^2, \d [0,1]^2)$, we have that the pair $W_{L}(K_{L}, K_{R})$ excises to
    \[
    W_L \times_{W_{\St e} }\left(\Sigma_{\Lk e} \times ([0,1]^2, \d [0,1]^2 ) \right).
    \]

    Note that $C_{L}$ acts trivially on $[0,1]^2$, hence, by the K\"unneth formula $b_{i+1}^{C_{L}}(W_{L} (K_{L},K_{R}); \DCL)=b_{i-1}^{C_{L}}(W_L\times_{W_{\St_{e}}} \Sigma_{\Lk e} ; \DCL) $, which vanishes by Lemma~\ref{l:ind} and hypothesis, and therefore, from the long exact sequence, $b_{i}^{C_{L}}(W_L K_R; \DCL) \leq b_{i}^{C_{L}}(\gS_{L} ; \DCL) $.

    Furthermore, since $W_LK_R=\gS_{R}/\ker \phi_{RL}$ and $\ker \phi_{RL} <C_R$, Lemma~\ref{l:universality} implies $ b_{i}^{C_{R}}(\gS_{R} ; \cD_{\F C_{R}}) \leq b_{i}^{C_{L}}(W_LK_R; \cD_{\F C_L}) $.
    Thus $ b_{i}^{C_{R}}(\gS_{R} ; \cD_{\F C_{R}}) \leq b_{i}^{C_{L}}(\gS_{L} ; \DCL) $, and the theorem follows by normalizing by the index, since the commutator subgroups have the same index $[W_{R}: C_{R}]=[W_{L}: C_{L}]=2^{\abs{L^{(0)}}}$.
\end{proof}

\section{\texorpdfstring{$\F$-$L^2$}{F-L2}-homology of RACG's based on subdivisions}\label{s:main}

We now use the decompositions of $\Sigma_L$ in Lemma \ref{l:v-remove} and Theorem \ref{t:subdivideedge} to compute $b_\ast^{(2)}(W_L; \F)$ for certain flag complexes $L$ arising from subdivision procedures.

\subsection{Edge subdivisions}

If $L$ is a simplicial complex and $e$ is an edge of $L$, the edge subdivision of $L$ along $e$, denoted $\Sub_e(L)$, is obtained by replacing the open star of $e$ with the cone on the suspended link of $e$.
The suspension points correspond to the endpoints of $e$, and we think of the cone point as the new midpoint $v_e$ of $e$.
It is easy to see that this edge subdivision of a flag complex is also flag.

We say that $L'$ is an iterated edge subdivision of $L$ if $L'$ can be obtained from $L$ by a sequence of edge subdivisions (the subdivided edges in this sequence do not have to be original edges of $L$).
If $K$ is a subcomplex of $L$ and $K'$ is an iterated edge subdivision of $K$, then there is an induced iterated edge subdivision $L'$.
One annoyance is that in general $L'$ does not depend only on the subdivision $K'$ but also on the ordering of subdivided edges.
An easy case where this problem disappears is when $L = K \ast J$, since subdividing edges of $K$ maintains the join structure, so $L' = K' \ast J$.
This generalizes easily to when $K$ is a full subcomplex:
\begin{lemma}\label{l:iesind}
    Suppose that $K$ is a full subcomplex of $L$.
    Let $K'$ be an iterated edge subdivision of $K$, and $L'$ the induced iterated edge subdivision of $L$.
    Then $L'$ depends only on $K'$.
\end{lemma}
\begin{proof}
    It suffices to show that any simplex $\gs$ of $L$ is subdivided identically in $L'$.
    Since $K$ is full, the intersection $\gs \cap K$ is a sub-simplex $\tau$.
    Any choice of edge subdivisions of $K$ which produces $K'$ obviously produces the same final subdivision $\tau'$ of $\tau$, so $\gs$ is subdivided into the join $\tau' \ast (\gs - \tau)$.
\end{proof}

The following lemma shows that the links of simplices in a subdivided complex are mostly unchanged, however some get subdivided and some get suspended.
\begin{lemma}
    [cf.
    \cite{cd95b}*{Section 5.3.3}]\label{l:subcases}
    Let $\gs$ be a simplex in $\Sub_e(L)$.
    There are four cases to consider:
    \begin{enumerate1}
        \item If $\gs \in \Lk_{ L}e$, then $\Lk_{ \Sub_e(L)}\gs = \Sub_e(\Lk_{ L}\gs)$.
        \item If $\gs = v_{e} \ast \gt$ for $\gt \in \Lk_{L} e$, then $\Lk_{ \Sub_e(L)}\gs = S^0 \ast \Lk_{ L}(e*\gt)$.
        \item If $\gs = [v_{e},v^\pm] * \gt$ for $\gt \in \Lk_{L} e$, then $\Lk_{ \Sub_e(L)}\gs = \Lk_{ L}(e*\gt )$.
        \item For all other simplices $\gs$, $\Lk_{ \Sub_e(L)}\gs = \Lk_{ L}\gs$.
    \end{enumerate1}
\end{lemma}

It is convenient to phrase Theorem \ref{t:subdivideedge} in terms of edge subdivisions rather than edge removals.
For the purposes of $\F$-$L^2$-Betti numbers there is no difference between the two.
\begin{lemma}
    Let $e$ be an edge in a flag complex $L$.
    Then
    \[
    b_*^{(2)}(W_{\Sub_e(L)}; \F) = b_*^{(2)}(W_{L-e}; \F).
    \]
\end{lemma}
\begin{proof}
    Let $v$ be the midpoint of the subdivided edge.
    Its link $\Lk_{\Sub_e(L)} v = S\Lk_{L} e $ is $\F$-$L^2$-acyclic since it is a suspension.
    Removing this vertex from $\Sub_e(L)$ produces $L-e$ and does not change $\F$-$L^2$-Betti numbers by Lemma~\ref{l:v-remove}.
\end{proof}
Hence, Theorem \ref{t:subdivideedge} implies:
\begin{corollary}\label{c:vanishingsub}
    If $L$ is a flag complex with $b_i^{(2)}(W_{L}; \F) = 0$ and $e$ is an edge with $b_{i-1}^{(2)}(W_{\Lk e}; \F) = 0$, then $b_i^{(2)}(W_{\Sub_e(L)}; \F) = 0$.
\end{corollary}
\begin{corollary}\label{c:concentrated+greater}
    Let $L'$ be an iterated edge subdivision of a flag complex $L$, and $n \in \mathbb{N}$.
    \begin{enumerate1}
        \item\label{i:acyclic}
        If for all odd-dimensional simplices $\gs^{2k-1} \in L$ (including $\gs = \emptyset$) we have $b_{*}^{(2)}(W_{\Lk \gs^{2k-1}}; \F) = 0$, then $b_{*}^{(2)}(W_{L'}; \F) = 0$.

        \item\label{i:concentration}
        If for all odd-dimensional simplices $\gs^{2k-1} \in L$ (including $\gs = \emptyset$) we have $b_{\neq n-k}^{(2)}(W_{\Lk \gs^{2k-1}}; \F) = 0$, then $b_{\neq n}^{(2)}(W_{L'}; \F) = 0$.

        \item\label{i:greater}
        If for all odd-dimensional simplices $\gs^{2k-1} \in L$ (including $\gs = \emptyset$) we have $b_{> n-k}^{(2)}(W_{\Lk \gs^{2k-1}}; \F) = 0$, then $b_{> n}^{(2)}(W_{L'}; \F) = 0$.
    \end{enumerate1}
\end{corollary}
\begin{proof}
    Since $\Lk_{\Lk_{\gt}}\gs=\Lk_{L}(\gt * \gs)$, it follows from Lemma~\ref{l:subcases} and Corollary~\ref{c:vanishingsub} that in all cases a single edge subdivision preserves the conditions on $L$.
\end{proof}
\begin{corollary}

    $b_{\ne n}(W_{b\d\Delta^{2n}}; \F) = 0$.
\end{corollary}
\begin{proof}
    Any barycentric subdivision is obtained via a sequence of edge subdivisions, see e.g.
    \cite{ln16}*{Proposition~3.1} (or Lemma~\ref{l:pies} below for a more general statement).
    Corollary \ref{c:concentrated+greater}\eqref{i:greater} implies that $b_{>n}^{(2)}(W_{b\Delta^{2n}}; \F) = 0$.
    Since $b\Delta^{2n}$ is a cone on $b\d \Delta^{2n}$, the statement follows from this vanishing and Poincar\'e duality.
\end{proof}

If $\F = \Q$, this confirms the first part of Theorem \ref{t:main} in the even dimensional case.
Another consequence of Corollary~\ref{c:concentrated+greater} is that edge subdivisions alone cannot produce a counterexample to the Singer conjecture.
\begin{corollary}
    Suppose that $L$ is a flag triangulation of $S^{n}$ such that $P_L$ and $P_{\Lk \gs}$ satisfy the Singer conjecture for all links of odd-dimensional simplices.
    Then for any iterated edge subdivision $L'$ of $L$, $P_{L'}$ also satisfies the Singer conjecture.
\end{corollary}

In particular, this shows that any iterated edge subdivision of the $n$-fold join of $2$ points with itself $O^{n}$ (the boundary of the $n$-octahedron) satisfies the Singer conjecture.
In the next subsection we will show that $b\d\Delta^{n}$ is such, and hence the first part of Theorem \ref{t:main} holds.
In fact, we will prove that more general relative barycentric subdivisions of $\d\Delta^{n}$ are iterated edge subdivisions of $O^{n}$.

\subsection{Relative barycentric subdivisions}\label{s:subdivisions}

Given a simplicial pair $(L,K)$, we can form the \emph{relative barycentric subdivision of $L$ relative to $K$}, $b(L,K)$, by barycentrically subdividing simplices of $L$ not in $K$.
More precisely, we construct $b(L,K)$ by induction on dimension, the 0-th stage being the set of vertices of $L$.
To construct the $i$-th stage from the $(i-1)$-st stage, for each simplex $\gs\in L^{(i)}$, if $\gs\in K$ we leave it unchanged, otherwise we introduce a new vertex $v_\gs$, and cone $b(\d\gs, \d\gs \cap K)$ to $v_\gs$.
Note that $b(L, \emptyset)=b(L, pt)$ is the usual barycentric subdivision of $L$, which we shorten to $b(L)$.

Combinatorially, vertices of $b(L,K)$ are simplices of $L \smallsetminus K$ together with vertices of $K$, and simplices of $b(L,K)$ are chains of simplices of $L$ with at most one simplex in $K$.
Note that $K$ is a full subcomplex of $b(L,K)$, and if $K$ is flag, then $b(L,K)$ is also flag.

The following lemma is well-known in the case of barycentric subdivisions.
\begin{lemma}\label{l:links}
    Let $b(L,K)$ be the relative barycentric subdivision of a simplicial pair $(L,K)$.
    If $v$ is a vertex that is the barycenter of a simplex $\gs \in L - K$ then
    \[
    \Lk_{ b(L,K)}v = b(\d \sigma, \d \gs \cap K) \ast b \Lk_L \gs.
    \]
    If $v$ is a vertex of $K$, then
    \[
    \Lk_{ b(L,K)}v = b(\Lk_{L} v, \Lk_{K}v).
    \]
\end{lemma}
\begin{proof}
    If $v$ is the barycenter of $\gs$, simplices of $\Lk_{ b(L, K)}v$ correspond to chains of simplices of $L$ with at most one simplex in $K$ such that the addition of $\gs$ is still a chain.
    Every such chain is a union of two chains, one contained in $\gs$ and the other containing $\gs$.
    Since $\gs$ is not in $K$, any simplex in the chain that is in $K$ must be contained in $\gs$.
    Therefore, the chains contained in $\gs$ correspond to $b(\d \sigma, \d \gs \cap K)$, the chains containing $\gs$ correspond to $\Lk_L \gs$, and the join structure is clear.

    For $v \in K$, $\Lk_{ b(L,K)}v$ consists of chains of simplices of $L$ properly containing $v$ with at most one simplex in $K$, which is precisely a simplex in $b(\Lk_L v, \Lk_{K}v)$.
\end{proof}

The next lemma follows immediately from the definition of $b(L,K)$.
\begin{lemma}
    If $L = L_1 \cup_{L_0} L_2$, $K \subset L$, and $K_i = L_i \cap K$, then
    \[
    b(L,K) = b(L_1,K_1) \cup_{b(L_0,K_0)} b(L_2,K_2).
    \]
\end{lemma}

\subsection{Relative barycentric subdivisions via edge subdivisions}
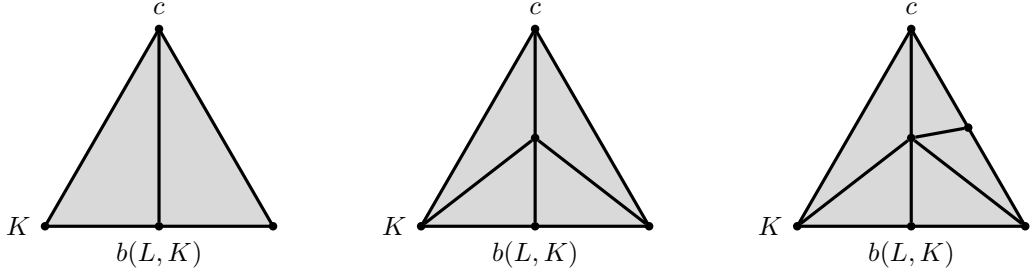
\begin{figure}
    \centering
    \usetikzlibrary{shapes.geometric, calc, positioning} \tikzset{ vertex/.style= {color = black, draw, fill = black, inner sep = 0, minimum size=2pt, shape=circle }, triangle/.style= {color = black, shape=regular polygon, regular polygon sides=3, minimum size=3.5cm, draw, fill=gray!30, very thick, outer sep=0 }, every pin edge/.style={<-,shorten <=1pt, thin } }
    \begin{tikzpicture}[very thick]
        \node[name=s, triangle] at (0,0) {}; \foreach \anchor/\placement/\a/\label in {corner 1/above/c/c, corner 2/left/a/K, corner 3/right/b/, 270/below/ab/{b(L,K)}}
        \draw (s.\anchor) node[vertex, label=\placement:$\label$ ] (\a) {};
        \draw (c)--(ab);

        \node[name=s, triangle] at (5,0) {}; \foreach \anchor/\placement/\a/\label in {corner 1/above/c/c, corner 2/left/a/K, corner 3/right/b/, 270/below/ab/{b(L,K)}, {center)+(0,0.3}/left/abc/}
        \draw (s.\anchor) node[vertex, label=\placement:$\label$ ] (\a) {};
        \draw (c)--(ab);
        \draw (a)--(abc);
        \draw (b)--(abc);

        \node[name=s, triangle] at (10,0) {}; \foreach \anchor/\placement/\a/\label in {corner 1/above/c/c, corner 2/left/a/K, corner 3/right/b/, 270/below/ab/{b(L,K)}, {center)+(0,0.3}/left/abc/, 30/right/bc/}
        \draw (s.\anchor) node[vertex, label=\placement:$\label$ ] (\a) {};
        \draw (c)--(ab);
        \draw (a)--(abc);
        \draw (b)--(abc);
        \draw (bc)--(abc);
    \end{tikzpicture}
    \caption{ $b(cL, cK)$ is an iterated edge subdivision of $cb(L, K)$.
    The leftmost figure is the cone on a relative barycentric subdivision of the bottom (edge, vertex).
    Subdividing the two edges as shown yields the relative barycentric subdivision $b(\Delta^2, \Delta^1)$.
    }
\end{figure}

Denote by $cL$ a cone on $L$ with cone vertex $c$.
\begin{lemma}\label{l:psubdiv}
    $b(cL,cK)$ is an iterated edge subdivision of $cb(L,K)$, obtained by subdividing the edges connecting $c$ to vertices in $b(L,K) \smallsetminus K$ in order of decreasing dimension.
\end{lemma}
\begin{proof}
    First, note that if we have a collection of edges with disjoint open stars, then we can do the edge subdivision on them in any order and get the same complex in the end.
    We claim that this is the case for our procedure at each stage.
    For two edges starting at $c$, the condition of having disjoint open stars is equivalent to their endpoints not being adjacent, which holds for vertices in $b(L,K) \smallsetminus K$ of the same dimension, and this disjointness is obviously preserved under subdivision.

    We now proceed by induction on the largest dimension of simplices of $L \smallsetminus K$.
    In the base case $L=K$ and both subdivisions are just $cL$ so we are not subdividing anything.
    Now, assume $(L',K)$ is a pair, and let $L$ be $L'$ with the top dimensional simplices of $L' \smallsetminus K$ removed.

    Now, for each top dimensional simplex $\gs$ of $L' \smallsetminus K$, subdivide the edge $e$ in $cb(L',K)$ connecting $c$ to the barycenter of $\gs$ in $b(L', K)$.
    This complex is obtained from $cb(L,K)$ by coning off $b(\d \gs, \d \gs \cap K)$ to $v_\gs$, forming $b(\gs, \gs \cap K)$, and then coning off $cb(\d \gs, \gs \cap K) \cup b(\gs, \gs \cap K)$ to $v_{c\gs}$.

    Now, by induction the rest of the edge subdivisions turn the pair $cb(L,K)$ into $b(cL,cK)$.
    The vertices $v_\gs$ are not contained in the closed star of any such edge, so by Lemma~\ref{l:subcases}(4) the link of $v_\gs$ is unchanged.
    By repeated application of Lemma~\ref{l:subcases}(1) and (4), the link of $v_{c\gs}$ gets subdivided, and hence, by induction, becomes $b(c\d \gs, c(\d \gs \cap K)) \cup b(\gs, \gs \cap K)=b(\d(c\gs), \d(c\gs) \cap cK)$.

    So, the resulting complex is obtained from $b(cL,cK)$ by coning off $b(\d \gs, \d \gs \cap K)$ to $v_\gs$, and coning off $b(\d(c\gs), \d(c\gs) \cap cK)$ to $v_{c\gs}$ for each top simplex $\gs$, which is precisely $b(cL',cK)$.
\end{proof}
\begin{lemma}\label{l:pies}
    If $K$ is a flag subcomplex of $L$, then $b(L,K)$ can be obtained by an iterated edge subdivision of $L$.
\end{lemma}
\begin{proof}
    We proceed by induction on the number of simplices of $L \smallsetminus K$.
    The inductive claim is that we have a sequence of edge subdivisions which turns $L$ into $b(L,K)$ such that the following two conditions are satisfied:
    \begin{enumeratei}
        \item For any simplex $\gs \in L\smallsetminus K$ its barycenter is obtained as the midpoint of the edge between the barycenters of two opposite faces of $\gs$.
        \item If $\gt$ and $\gr$ are two intersecting faces of $\gs \in L\smallsetminus K $ which span $\gs$, then the barycenter of $\gs$ appears in the sequence no later than one of the barycenters of $\gt$, $\gr$.
    \end{enumeratei}
    In the base case $L=K$ and the partial subdivision is just $L$, so we are not subdividing anything.
    Now, assume $(L',K)$ is a pair with $L'\neq K$ and $K$ flag, let $\gs$ be a maximal simplex of $L' \smallsetminus K$, and let $L=L'-\gs$.
    Since $K$ is flag, there exists a facet of $\d\gs$ which is not contained in $K$, and hence its barycenter is a vertex in $b(L,K)$.
    Its opposite barycenter is just a vertex of $L$ and hence a vertex of $b(L,K)$, so we conclude that $b(\d\gs, \d\gs \cap K)$ contains a pair of barycenters of opposite faces of $\d\gs$.

    By induction there is a sequence of edge subdivisions of $L$ which turn the pair $(L,K)$ into $b(L,K)$.
    Now perform the same sequence in $L'$.
    This clearly does not produce any vertices in the interior of $\gs$, and at some point of the sequence we will create an edge going through the interior of the simplex between the barycenters of two faces $\gt$ and $\gr$ for the first time.
    We claim that these faces are opposite to each other.
    Assume that the barycenter of $\gr$ appears before $\gt$ in the sequence.
    Since the edge goes through the interior, $\gt$ and $\gr$ span $\gs$, so we need to show that they are disjoint.
    By condition (i) the barycenter of $\gt$ is the midpoint of the barycenters of opposite faces $\ga$ and $\gb$ of $\gt$, who appear before $\gt$.
    Let $\gd$ be the span of $\ga$ and $\gr$; notice that $\delta \subset \d\gs$ as before this point there were no edges in the interior of $\gs$.

    If $\ga$ and $\gr$ intersect, then by condition (ii) the barycenter of $\gd$ appears no later than one of the barycenters of $\ga$, $\gr$, and hence before the barycenter of $\gt$.
    Since $\gd$ and $\gb$ span $\gs$, the edge between their barycenters goes through the interior of $\gs$, contradicting our choice of $\gt$ and $\gr$.
    So $\ga$ and $\gr$ are disjoint, and similarly $\gb$ and $\gr$ are also disjoint.
    This proves the claim.

    Now we modify the sequence by inserting the subdivision of this edge immediately after it appeared.
    This subdivision produces the barycenter of $\gs$ and affects only the interior of $\gs$, producing a cone on the subdivision of $\d\gs$ at the previous stage.
    Hence, the subsequent subdivisions will further subdivide $L$ into $b(L,K)$, $\d\gs$ into $b(\d\gs, \d\gs \cap K)$, and $\gs$ into the cone on $b(\d\gs, \d\gs \cap K)$, and therefore the end result of the new sequence is $b(L',K)$.
    It is also clear that the new sequence satisfies conditions (i) and (ii).
\end{proof}
\begin{lemma}\label{l:twosubs}
    Suppose we have complexes $J < K < L$ with $J$ flag.
    Then $b(L,J)$ is an iterated edge subdivision of $b(L,K)$.
\end{lemma}
\begin{proof}
    By Lemma \ref{l:pies}, we know that $b(K,J)$ is an iterated edge subdivision of $K$.
    We claim that the induced iterated edge subdivision of $b(L,K)$ is $b(L,J)$.
    Since $K$ is a full subcomplex of $b(L,K)$, arguing as in Lemma~\ref{l:iesind}, the induced iterated edge subdivision of a simplex $\gs$ in $b(L,K)$ decomposes as a join $b(\gt, \gt \cap J)*(\gs-\gt)$ where $\gt =\gs \cap K$, which is precisely what $\gs$ gets subdivided into in $b(L,J)$.
\end{proof}
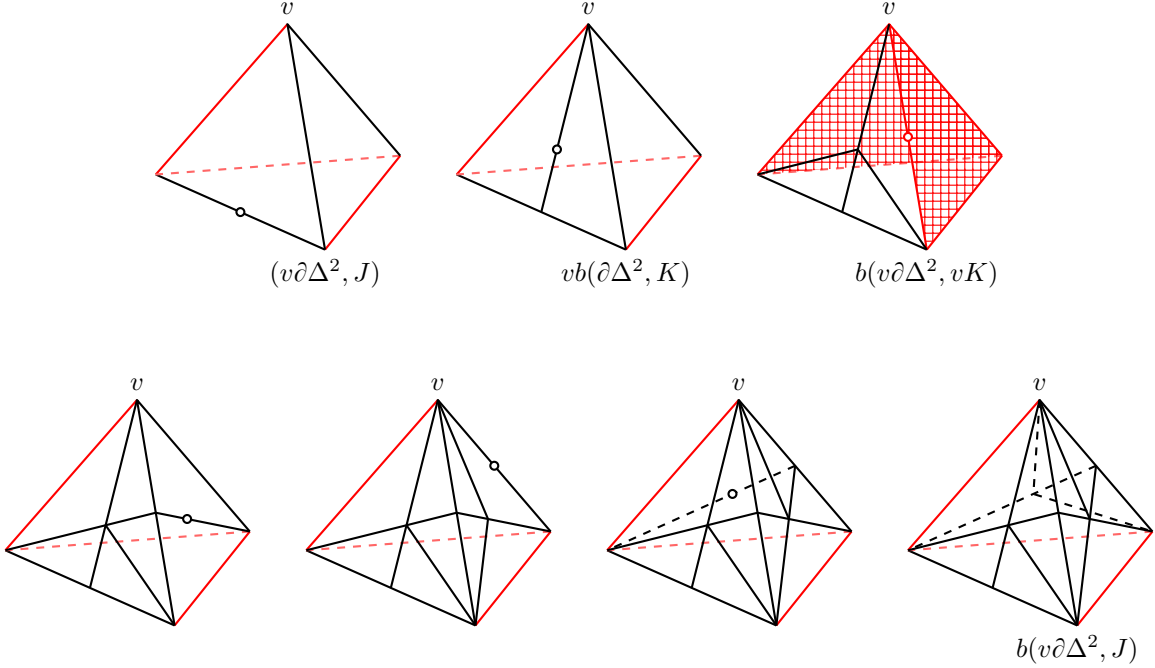
\begin{figure}
    \centering
    \begin{tikzpicture}[join=round,thick]

        \tikzstyle{deltaffill} = [color=black!70,fill=gray,fill opacity=0.7] \tikzstyle{deltafill} = [opacity=1, fill=lightgray!40,fill opacity=0.8] \tikzstyle{baryfill} = [fill=black] \tikzstyle{vertexfill} = [fill=white, color=white] \tikzset{ sub/.style= {color = black, draw, fill = white, inner sep = 0, minimum size=3pt, shape=circle }}

        \def\tetrax#1{
        \coordinate (p1) at (1,.25);
        \coordinate (p2) at (-.5,2);
        \coordinate (p3) at (-2.25,0);
        \coordinate (p4) at (0,-1);

        \foreach \x/\y in {1/2, 1/3, 1/4, 2/3, 2/4, 3/4}
        \coordinate (p\x\y) at (barycentric cs:p\x=1,p\y=1);

        \foreach \x/\y/\z in {1/2/3, 1/2/4, 1/3/4, 2/3/4}
        \coordinate (p\x\y\z) at (barycentric cs:p\x=1,p\y=1,p\z=1);

        \draw (p4)--(p3);
        \draw[red] (p2)--(p3);

        \draw[dashed,color=red!60] (p1)--(p3);

        \draw[color=red] (p1)--(p4);
        \node[above] at (p2){$v$};
        \node[below] at (p4) {$#1$}; }

        \def\tetrar#1{ \tetrax{#1}
        \fill [red, pattern = grid, pattern color = red, opacity = .25 ] (p1)--(p2)--(p3)--cycle;
        \fill [red, pattern = grid, pattern color = red, opacity = .25 ] (p1)--(p2)--(p4)--cycle;
        \draw[red] (p1)--(p2)--(p4);

        }
        \def\tetrab#1{ \tetrax{#1}
        \draw (p1)--(p2)--(p4); }
        \begin{scope}[xshift = 2cm] \tetrab{(v\d \gD^2, J)}
            \node[sub] at (p34){};
        \end{scope}
        \begin{scope}[xshift = 6cm] \tetrab{vb(\d \gD^2, K)}
            \draw (p34)--(p2);
            \node[sub] at (p234){};
        \end{scope}
        \begin{scope}[xshift = 10cm] \tetrar{b(v\d \gD^2, v K)}
            \draw (p34)--(p2);
            \draw (p3)--(p234)--(p4);
            \node[sub, color=red, fill=white] at (p24){};
        \end{scope}
        \begin{scope}[xshift = 0cm, yshift = -5cm] \tetrab{}
            \draw (p34)--(p2);
            \draw (p3)--(p234)--(p4);
            \draw (p234)--(p24)--(p1);
            \node[sub] at (p124){};
        \end{scope}
        \begin{scope}[xshift = 4cm, yshift = -5cm] \tetrab{}
            \draw (p34)--(p2);
            \draw (p3)--(p234)--(p4);
            \draw (p234)--(p24)--(p1);
            \draw (p2)--(p124)--(p4);
            \node[sub] at (p12){};
        \end{scope}
        \begin{scope}[xshift = 8cm, yshift = -5cm] \tetrab{}
            \draw (p34)--(p2);
            \draw (p3)--(p234)--(p4);
            \draw (p234)--(p24)--(p1);
            \draw (p2)--(p124)--(p4);
            \draw (p124)--(p12);
            \draw[dashed] (p12)--(p3);
            \node[sub] at (p123){};
        \end{scope}
        \begin{scope}[xshift = 12cm, yshift = -5cm] \tetrab{b(v\d \gD^2, J)}
            \draw (p34)--(p2);
            \draw (p3)--(p234)--(p4);
            \draw (p234)--(p24)--(p1);
            \draw (p2)--(p124)--(p4);
            \draw (p124)--(p12);
            \draw[dashed] (p12)--(p3);
            \draw[dashed] (p2)--(p123)--(p1);
        \end{scope}
    \end{tikzpicture}
    \caption{Our subdivision procedure turning the cone $vb(\d \Delta^2,K)$ into $b(v\d \Delta^2,J)$.
    Here, $J$ is the union of the three red edges, and $K=J \cap \Delta^2$.
    The first row depicts subdivisions of Lemma~\ref{l:psubdiv}, and the second --- of Lemma~\ref{l:twosubs}.
    Each edge marked by $\circ$ is subdivided in the next picture.
    }
\end{figure}
\begin{Theorem}\label{c:vanishingtomei}
    Let $J$ be a flag subcomplex of $\d\gD^{n}$.
    Then $b(\d\gD^{n}, J)$ can be constructed as an iterated edge subdivision of the suspension $Sb(\d \gt, J\cap \d\gt))$ for some facet $\gt$ of $\d\gD^{n}$.
\end{Theorem}
\begin{proof}
    Since $J$ is flag, it is contained in the star of some vertex $v$ of $\d\gD^{n}$.
    Therefore, the barycenter $v_{\gt}$ of its opposite facet $\gt$ is in $b(\d\gD^{n}, J)$.
    Let $L=\d\gt$ and $K=J\cap \d\gt$.

    By Lemma~\ref{l:twosubs} applied to the triple $(v L, vK, J)$, $b(vL, J)$ is an iterated edge subdivision of $b(vL, vK)$, which by Lemma~\ref{l:psubdiv} is an iterated edge subdivision of $vb(L, K)$.
    Therefore,
    \begin{align*}
        b(\d\gD^{n}, J)=b(v L\cup_{L} v_{\gt} L, J) &= b(vL, J) \cup_{b(L, K)} v_{\gt} b(L, K) \\
        \intertext{is an iterated edge subdivision of } Sb(\d \gt, J\cap \d\gt)=Sb(L,K) &=v b(L, K) \cup_{b(L, K)} v_{\gt} b(L, K).
    \end{align*}
\end{proof}
\begin{corollary}\label{c:octahedron}
    Let $K$ be a flag subcomplex of $\d \Delta^{n}$.
    Then $b(\d\gD^{n}, K)$ is an iterated edge subdivision of the boundary $O^n$ of the $n$-octahedron.
\end{corollary}
\begin{proof}
    Let $\Delta^{n-1}$ be a face of $\Delta^{n}$ that is not contained in $K$.
    By induction, $b(\d\gD^{n-1}, K \cap \d \Delta^{n-1})$ is an iterated edge subdivision of $O^{n-1}$; hence the induced edge subdivision of $O^n$ is $Sb(\d\gD^{n-1}, K \cap \d \Delta^{n-1})$.
    Now apply Theorem \ref{c:vanishingtomei}.
\end{proof}
\begin{remark}
    More general versions of Lemmas \ref{l:pies} and \ref{l:twosubs} and Corollary \ref{c:octahedron} follow from earlier work of Feichtner--M\"{u}ller \cite{fm04}, Postnikov--Reiner--Williams \cite{pnw08}, and Volodin \cite{v10}.
    Given a simplicial complex $L$, a \emph{building set} is a collection of simplices $\mathcal{B}$ containing $L^{(0)}$ such that if $\gs, \gt \in \mathcal{B}$ with $\gs \cup \gt \in L$ then $\gs \cup \gt \in \mathcal{B}$ if $\gs \cap \gt \ne \emptyset$.
    Any building set determines a corresponding nested set complex, which ends up being a subdivision of $L$ where new vertices correspond to elements in $\mathcal{B} \smallsetminus L^{(0)}$.
    If $K \subset L$ is a subcomplex, it is easy to see that the vertices of $L$ together with all simplices not contained in $K$ forms a building set, and the corresponding nested set complex is isomorphic to $b(L, K)$.

    Feichtner and M\"{u}ller \cite{fm04}*{Theorem 4.2} show that if $\mathcal{B'} \subset \mathcal{B}$ are two building sets, then the nested complex for $\mathcal{B}$ is an iterated stellar subdivision of the nested complex for $\mathcal{B'}$.
    Volodin \cite{v10}*{Lemma 6} shows that if the nested set complex for $B$ is flag then this iterated stellar subdivision can be chosen to be an iterated edge subdivision.
    (Volodin only considers $L = \d \Delta^n$ and assumes both nested set complexes are flag, but his argument goes through in this more general setup.) Finally, Postnikov, Reiner and Williams \cite{pnw08}*{Section 7} show that any building set for $L = \d \Delta^n$ with flag nested set complex contains a building set whose nested set complex is isomorphic to $O^n$.

    The dual polytopes of nested set complexes associated to building sets for $\d \Delta^n$ are called \emph{nestohedra}, they include permutohedra and associahedra, see \cite{pnw08}*{Section 10} for these and some other examples.
\end{remark}

We can now complete the proof of the first part of our main theorem from the introduction (which is the special case below when $K = \emptyset$).
\begin{Theorem}\label{t:vanishingtomei}
    Let $K$ be a flag subcomplex of $\gD^{2n+1}$.
    Then $b_\ast^{(2)}(W_{b(\d \Delta^{2n+1}, K)}; \F) = 0$.
\end{Theorem}
\begin{proof}
    This follows immediately from Corollary \ref{c:octahedron} and Corollary \ref{c:concentrated+greater} \eqref{i:acyclic}.
\end{proof}

Using Theorem \ref{t:vanishingtomei}, we get the following general statement on relative barycentric subdivisions.
\begin{Theorem}\label{t:partialsinger}
    Suppose $L$ is a complex of dimension $\leq 2n-1$, and $K$ is a flag subcomplex.
    Then
    \[
    b_{i}^{(2)}(W_{b(L,K)}; \F) = b_{i}^{(2)}(W_K; \F) \text{ for } i> n,
    \]
    and
    \[
    b_{n}^{(2)}(W_{b(L,K)}; \F) \geq b_{n}^{(2)}(W_K; \F).
    \]
\end{Theorem}
\begin{proof}
    The barycenter of an odd-dimensional simplex $\gs$ in $L \smallsetminus K$ has link equal to $b(\d\gs, \d\gs \cap K) \ast b\Lk_{ L}\gs$ by Lemma \ref{l:links}.
    Since $K$ is flag, $\d \gs \cap K$ is flag, so Theorem \ref{t:vanishingtomei} guarantees that the first factor is $\F$-$L^2$-acyclic, hence the link is $\F$-$L^2$-acyclic by the K\"{u}nneth formula, so we can remove these vertices from $b(L, K)$ without changing $b_{*}^{(2)}(\cdot ; \F)$.
    Furthermore, if we remove the barycenters of odd-dimensional simplices in order of decreasing dimension, then the first factor in the link is unchanged at each stage, so by Lemma~\ref{l:v-remove} we can remove all the odd-dimensional barycenters.

    At this point, we have a subcomplex of $b(L,K)$ spanned by the vertices of $K$ and the even-dimensional barycenters of $L \smallsetminus K$.
    We now remove the even-dimensional barycenters in order of increasing dimension to obtain $K$.

    For a $0$-dimensional barycenter $v$ (a vertex of $L\smallsetminus K$), the new link is a subcomplex of $b\Lk_{ L}v$ spanned by the even-dimensional barycenters of dimension $\ge 2$.
    Therefore, this subcomplex is at most $(n-2)$-dimensional, hence the link has trivial $b_i^{(2)}(\cdot; \F)$ for $i > n - 1 $, hence removing these vertices does not affect $b_{>n}^{(2)}(\cdot ; \F)$ and decreases $b_{n}^{(2)}(\cdot ; \F)$.
    Now, if $v_\gt$ is an even-dimensional barycenter of a simplex $\gt$ with $\dim \gt \geq 2$, in our procedure we have removed all barycenters of faces of $\gt$ except those in $K$.
    Therefore, the link of $v$ decomposes as a join of $\gt' = \gt \cap K$ with a subcomplex of $b\Lk_{ L}\gt$ spanned by even-dimensional barycenters.
    This subcomplex is at most $(n-2 - \dim \gt/2)$-dimensional, and, since $W_{\gt'}$ is finite, removing these vertices again does not affect $b_{\geq n}^{(2)}(\cdot ;\F)$.
\end{proof}

Combining Theorem \ref{t:partialsinger} with Poincar\'e duality proves a more general version of the second part of our main theorem from the introduction:
\begin{corollary}\label{c:vanishingbs}
    Suppose $L$ is a barycentric subdivision of a triangulation of $S^{n-1}$.
    \begin{enumerate1}
        \item If $n$ is even, then $b_{\ast}^{(2)}(W_{L}; \F) = 0$ for $\ast \ne \frac{n}{2}$.
        \item If $n$ is odd then $b_{\ast}^{(2)}(W_{L}; \F) = 0$ for $\ast \ne \frac{n-1}{2}, \frac{n+1}{2}$.
    \end{enumerate1}
\end{corollary}
\begin{remark}
    Note that in Corollary \ref{c:vanishingbs} we make no assumptions on the triangulation of $S^n$, in particular it does not have to be piecewise-linear.
    In fact, if we start with a triangulation of a generalized homology sphere $GHS^n$ (by definition a homology $n$-manifold with the same homology as $S^n$), then the Davis complex is a homology $n$-manifold and Poincar\'e duality combined with Theorem \ref{t:partialsinger} implies the same vanishing result.
\end{remark}

\section{Torsion growth}\label{s:torsion}

In this section we give explicit counterexamples to the Torsion concentration conjecture from the introduction.

\subsection{Universal coefficients and torsion growth}

Given a group $G$, and a residual chain $G_k$ of finite index normal subgroups, let
\[
t_i^{(2)}(G) := \limsup \frac{\logtor H_{i}(G_k;\zz)} {[G:G_k]}.
\]
This may depend on the chain, but we omit it from the notation.

We will also let $t_i(G)_p := \dim_{\F_p} \Tor( H_i(G; \zz); \Fp)$ denote the number of $\zz_{p^k}$-summands in $H_i(G; \zz)$, and set
\[
t_i^{(2)}(G)_p := \limsup \frac{t_i(G_{k})_p} {[G:G_k]}.
\]
Note that
\[
t_i^{(2)}(G) \ge \max_{p} t_i^{(2)}(G)_p \log p.
\]

Recall that $\mathcal{P}(G)$ denotes the poset of finite index normal subgroups $H < G$ with $G/H$ a $p$-group.
\begin{lemma}\label{l:torsion}
    Suppose that $G$ is residually torsion-free nilpotent, and $b_{n+1}^{(2)}(G; \F_p) > b_{n+1}^{(2)}(G, \Q)$.
    Then, for any normal residual chain $G_k \nsgp G$ of finite index subgroups,
    \[
    t_n^{(2)}(G) + t_{n+1}^{(2)}(G) > 0.
    \]
    If additionally we have $b_{n+2}^{(2)}(G; \F_p) = 0$, then
    \[
    t_{n}^{(2)}(G) > 0
    \]
    for any cofinal chain of subgroups in $\mathcal{P}(G)$.\footnote{This is not just vacuously true: cofinal chains in $\mathcal{P}(G)$ do exist.
    For instance, one can take $G_k$ to be the intersection of all subgroups of index $p^i$ for $i \le k$.}
\end{lemma}
\begin{proof}
    For any normal residual chain of finite index subgroups, L\"{u}ck's approximation theorem implies $b_{n+1}^{(2)}(G; \Q) = \lim_k \frac{b_{n+1}(G_k; \Q)} {[G:G_k]}$, and Theorem~\ref{t:binf} implies $b_{n+1}^{(2)}(G; \F_p) \le \frac{b_{n+1}(G_k; \F_p)}{[G:G_k]}$.
    The universal coefficient theorem gives us that
    \[
    t_n(G_{k})_p + t_{n+1}(G_{k})_p = b_{n+1}(G_{k};\Fp) - b_{n+1}(G_{k};\Q).
    \]
    Hence, taking $\limsup$ and using the assumption gives $t_n^{(2)}(G)_p + t_{n+1}^{(2)}(G)_p > 0$, and therefore $ t_n^{(2)}(G) + t_{n+1}^{(2)}(G) > 0$.

    Now, let $\{G_{k}\}$ be a cofinal chain in $\mathcal{P}(G)$.
    Note that $\{G_{k}\}$ is residual, since $G$ is residually $p$-finite.
    Assuming that $b_{n+2}^{(2)}(G; \F_p) = 0$, Theorem~\ref{t:bpinf} gives $\lim \frac{b_{n+2}(G_k; \F_p)} {[G:G_k]} = 0$.
    Then by the universal coefficient theorem in degree $n+2$, $t_{n+1}^{(2)}(G)_p = 0$, and therefore $t_{n}^{(2)}(G)_p > 0$, hence we get $t_{n}^{(2)}(G)> 0$.
\end{proof}

Combining this lemma with Theorem~\ref{t:partialsinger}, and noting that $b_{i}^{(2)}(W_K; \F) =0 $ for $i> \dim K +1$ gives the following corollary:
\begin{corollary}\label{c:torsion}
    Suppose $L$ is a complex of dimension $\leq 2n-1$, and $K$ is a flag subcomplex of dimension $n$ with $b_{n+1}^{(2)}(W_K; \Fp) > b_{n+1}^{(2)}(W_K; \Q)=0$.
    Then
    \begin{align*}
        b_{>n}^{(2)}(W_{b(L,K)}; \Q) &=0, \\
        b_{>n+1}^{(2)}(W_{b(L,K)}; \Fp) &=0, \\
        b_{n+1}^{(2)}(W_{b(L,K)}; \Fp) &>0, \\
        \intertext{and, for any cofinal chain in $\mathcal{P}(W_{b(L,K)})$,} t_{n}^{(2)}(W_{b(L,K)}) &> 0.
    \end{align*}
\end{corollary}

\subsection{Counterexamples to the Torsion concentration conjecture}

Another corollary is an improved version of the main examples in \cite{aos21}.
Let $p$ be an odd prime, and let $OL$ be the flag $3$-dimensional complex from \cite{aos21} that has $b_4^{(2)}(W_{OL}; \F_p) \ne 0$ and $b_*^{(2)}(W_{OL};\Q) = 0$, and which embeds as a full subcomplex into a triangulation of $S^6$.
Define a flag triangulation $T$ of $S^{6}$ by taking $T$ to be the relative barycentric subdivision $b(S^{6}, OL)$.

Since $OL$ is a full subcomplex of $T$, the join $OL \ast OL$ is a full $7$-dimensional subcomplex of $T \ast T \cong S^{13}$, and the triple join $OL*OL*OL$ is a full $11$-dimensional subcomplex of $T * T * T\cong S^{20}$,

Let $T'=b(T*T, OL \ast OL)$, and $T'' = b(T * T * T, OL*OL*OL)$.
So $T$, $T'$, and $T''$ are flag triangulations of $S^{6}$, $S^{13}$, and $S^{20}$, respectively.
\begin{Theorem}\label{t:fpsinger}
    \leavevmode
    \begin{itemize}
        \item $P_{T}$ violates the $\Fp$-version of the Singer conjecture.
        \item $P_{T'}$ satisfies the Singer conjecture, and violates the $\Fp$-version and the Torsion concentration conjecture.
        \item $P_{T''}$ violates both the $\Fp$-version of the Singer conjecture and the Torsion concentration conjecture.
    \end{itemize}
\end{Theorem}
\begin{proof}
    Applying Theorem~\ref{t:partialsinger} with $n=4$ to $T$ gives $b_{4}^{(2)}(W_T; \Fp) \ge b_{4}^{(2)}(W_{OL}; \Fp) \ne 0$, proving the first claim.

    By the K\"unneth formula, $b_8^{(2)}(W_{OL*OL}; \F_p) \ne 0$ and $b_8^{(2)}(W_{OL*OL}; \Q) = 0$.
    So, applying Corollary~\ref{c:torsion} with $n=7$ to $T'$ gives $b_{8}^{(2)}(W_{T'}; \Fp) \ne 0$, $b_{>7}^{(2)}(W_{T'}; \Q) = 0$, and $t^{(2)}_{7}(W_{T'}) \neq 0$.
    Also, by Poincar\'e duality, $b_{< 7}^{(2)}(W_{T'}; \Q) = 0$.
    This proves the second claim.

    Similarly, $b_{12}^{(2)}(W_{OL*OL*OL}; \F_p) \ne 0$ and $b_{12}^{(2)}(W_{OL*OL*OL}; \Q) = 0$.
    Applying Corollary~\ref{c:torsion} with $n=11$ to $T''$ gives $b_{12}^{(2)}(W_{T''}; \Fp) \ne 0$, and $t^{(2)}_{11}(W_{T''})\neq 0$.
\end{proof}

Taking joins with $5$-gons gives similar examples in higher dimensions of the same parity.
At this point, we cannot prove the Singer conjecture for any odd-dimensional non-$\F_p$-$L^{2}$-acyclic example.

\section{Minimally branching complexes}

We end with a non-manifold application of Theorem \ref{t:subdivideedge}.
\begin{Theorem}\label{t:trivalent}
    If $L$ is a connected flag graph with vertices of degree $\le 3$, then $b_2^{(2)}(W_L; \F) = 0$ unless $L = K_{3,3}$.
\end{Theorem}
\begin{proof}
    If $L$ has a vertex of degree $\le 2$, we can iteratively remove the vertices of $L$, at each stage removing a vertex of degree $\le 2$.
    These removals does not affect $b_2^{(2)}(; \F)$, hence $b_2^{(2)}(W_L; \F) = 0$.
    Therefore, we can assume that $L$ is trivalent.
    Furthermore, note that if $L$ has a proper full subcomplex $A$ so that $L - A$ is disconnected and $b_1^{(2)}(W_A; \F) = 0$, then we get a decomposition $L = L_1 \cup_A L_2$ where each $L_i$ has a vertex of degree $\leq 2$.
    Therefore, $b_2^{(2)}(W_{L_i}; \F) = 0$ so $b_2^{(2)}(W_L; \F) = 0$ by Mayer--Vietoris.
    In particular, we are done if $L$ has a separating edge or separating pair of non-adjacent vertices, so assume that this does not occur and that $L \ne K_{3,3}$.

    We now show that $L$ has two non-adjacent vertices $u$ and $v$ with non-separating stars.
    There are two cases: either $L$ contains a copy of a suspension of three points as a full subcomplex, or it does not.

    Suppose there is a suspension of three points inside of $L$ with $x$ and $y$ suspension points and base $\{a,b,c\}$.
    Let $\{a',b',c'\}$ denote the endpoints of edges starting at $\{a,b,c\}$ which are not in the suspension.
    Note that $\{a',b',c'\}$ are all different.
    Indeed, since $L \ne K_{3,3}$, they cannot be all the same, and if $a'=b'$, then $a'$ and $c'$ would separate the suspension from the rest of the complex.

    Next we claim that the stars of $\{a,b,c\}$ are all non-separating.
    Consider a vertex $z \in L - \St a - c $.
    Since the pair $a', c$ does not separate, there is a geodesic in $L-a'-c$ from $z$ to $b$.
    Such geodesic cannot go through $x$, $y$ and $a$, so $z$ is connected to $b$ in $L - \St a - c$.
    Thus $L - \St a - c$ is connected, and since $c$ is adjacent to $c'$, $L - \St a$ is connected as well.
    Therefore we can take $a$ and $b$ as $u$ and $v$.

    For the second case, we claim that we can take $u$ and $v$ to be points of maximum distance.
    The maximality of the distance implies that for any vertex in $\Lk v$ at least one of the edges from it points in the direction of $u$.
    Suppose, to the contrary, that $\St v$ separates, and let $C$ be a component of $L - \St v$ not containing $u$.
    Then the first edge of a geodesic in $L$ from any vertex of $C$ to $u$ must have the endpoint in $\Lk v$.
    Therefore, vertices in $\Lk v $ are adjacent to at most one vertex in $C$, and vertices in $C$ are all adjacent to a vertex in $\Lk v$.
    So, $C$ has at most three vertices.
    $C$ cannot be a single vertex, since $L$ does not contain suspensions of three points.
    If there are two vertices in $C$, then they are adjacent and each is adjacent to two vertices in $\Lk v$, so one vertex in the link is adjacent to both, contradiction.
    If there are three, then each is adjacent to a unique vertex in the link, since $L$ is trivalent they all must be adjacent, so we get a triangle, contradicting flagness.
    Thus stars of $u$ and $v$ do not separate, and we are done with the second case.

    Now, $L$ is a subcomplex of the suspension $S(L - u - v)$ which has $b_\ast^{(2)}(W_{S(L - u - v)}; \F) = 0$ (we think of $u$ and $v$ as the two suspension points).
    By repeated application of Theorem \ref{t:subdivideedge}, we will be done if we can iteratively remove the extra edges of $S(L - u -v)$ until we get to $L$, so that at each step, $b_{1}^{(2)}(W_{\Lk e}; \F) = 0$.

    We first remove edges containing $u$ which join $u$ to vertices $\notin \Lk_L u$, starting with an edge whose other endpoint is $w \in \Lk_{L}v \smallsetminus \Lk_{L}u$.
    (Note that since $u$ and $v$ are note adjacent and $L \neq K_{3,3}$ the links of $u$ and $v$ are different.) The link of this edge is the same as the link of $w$ in $L - u - v$, so it is $2$ points and hence $b_{\ast}^{(2)}(W_{\Lk e}; \F) = 0$.
    We continue removing all edges emanating from $u$ and ending outside of $\Lk_{L}u$.
    Since $L - \St u$ is connected, any vertex in $L - \St u$ is connected by a path to $v$, and hence connected to a vertex in $\Lk v$.
    Therefore, we can always remove these edges so that links are $\le 2$ points.
    Removing these edges containing $u$ does not change any of the links of edges containing $v$, so we do the same procedure starting at a vertex in $\Lk_{L}u \smallsetminus \Lk_{L}v$, until we are left with $L$.
\end{proof}

The higher dimensional version of Theorem \ref{t:trivalent} is easier to prove, due to the following lemma.
Let $3^{*n}$ denote the $n$-fold join of $3$ points with itself.
\begin{lemma}\label{l:join3}
    Suppose $L$ is an $n$-dimensional connected flag complex with links of vertices $\cong 3^{*n}$.
    If $n \ge 2$, then $L = 3^{*n+1}$.
\end{lemma}
\begin{proof}
    Let $v \in L$, and let $x \in \Lk v$.
    Then, since $\Lk(x) \cong 3^{*n}$, $x$ is adjacent to two vertices $u$ and $w$ that are not in $\St v$, and $\{v,u,w\}$ is a $3$ point factor of $\Lk(x)$.
    Hence, if $y \in \Lk(v) $ is adjacent to $x$, both $u$ and $w$ are adjacent to $y$.
    Since $\Lk v \cong 3^{*n}$ is connected for $n \ge 2$, it follows that $u$ and $w$ are the same for all vertices in $\Lk v$.
    So we get a copy of $3^{*n+1}$ in $L$, and since $L$ is connected we are done.
\end{proof}
\begin{Theorem}\label{t:annoying}
    Suppose that $L$ is an $(n-1)$-dimensional flag complex such that every $(n-2)$-simplex is contained in at most three $(n-1)$-simplices and links of $k$-simplices are connected for $-1 \leq k < n-2$.
    Then $b_{n}^{(2)}(W_L; \F) = 0$ unless $L=3^{*n}$.
\end{Theorem}
\begin{proof}

    The proof is by induction on $\dim L$.
    If $L$ is $0$-dimensional it is obvious and if $L$ is $1$-dimensional, then it is Theorem \ref{t:trivalent}.
    Now assume that the statement holds for all such $(n-1)$-dimensional flag complexes, and let $L \neq 3^{*n+1}$ be an $n$-dimensional flag complex with $n \ge 1$ which satisfies the hypotheses.
    Note that $L$ is connected and if $v$ is a vertex of $L$ then $\Lk v$ satisfies the conditions of the theorem, since the link of $\gs$ in $\Lk v$ is the link of $v \gs$ in $L$.

    We can assume that the link of each vertex in $L$ is $3^{*n}$.
    Otherwise by induction it would have $b_{n}^{(2)}(W_{\Lk v}; \F) = 0$, and hence its removal does not affect $b_{n+1}^{(2)}(W_L; \F)$ by Lemma \ref{l:v-remove}.
    So, we are done by Lemma \ref{l:join3}.
\end{proof}

\bigskip
\noindent Grigori Avramidi, Max Planck Institute for Mathematics, Bonn, Germany, 53111\\
\url{gavramidi@mpim-bonn.mpg.de} \\

\noindent Boris Okun, University of Wisconsin-Milwaukee, Department of Mathematical Sciences, PO Box 413, Milwaukee, WI 53201-0413\\
\url{okun@uwm.edu} \\

\noindent Kevin Schreve, Louisiana State University, Department of Mathematics, Baton Rouge, LA, 70806 \\
\url{kschreve@lsu.edu}

\end{document}